\documentclass[12pt]{amsart}
\usepackage[pdftex]{graphicx}
\usepackage{amsmath,amssymb,stmaryrd}
\usepackage{amsthm}
\usepackage{enumerate}
\usepackage{comment}
\usepackage{bm}
\usepackage[all]{xy}
\usepackage[pdftex]{color}
\usepackage{hyperref}
\usepackage[margin=30truemm]{geometry}

\usepackage{lineno}
\newcommand*\patchAmsMathEnvironmentForLineno[1]{
  \expandafter\let\csname old#1\expandafter\endcsname\csname #1\endcsname
  \expandafter\let\csname oldend#1\expandafter\endcsname\csname end#1\endcsname
  \renewenvironment{#1}
     {\linenomath\csname old#1\endcsname}
     {\csname oldend#1\endcsname\endlinenomath}}
\newcommand*\patchBothAmsMathEnvironmentsForLineno[1]{
  \patchAmsMathEnvironmentForLineno{#1}
  \patchAmsMathEnvironmentForLineno{#1*}}
\AtBeginDocument{
\patchBothAmsMathEnvironmentsForLineno{equation}
\patchBothAmsMathEnvironmentsForLineno{align}
\patchBothAmsMathEnvironmentsForLineno{flalign}
\patchBothAmsMathEnvironmentsForLineno{alignat}
\patchBothAmsMathEnvironmentsForLineno{gather}
\patchBothAmsMathEnvironmentsForLineno{multline}
}

\theoremstyle{theorem}
\newtheorem{thm}{Theorem}[section]
\theoremstyle{definition}
\newtheorem{dfn}[thm]{Definition}

\theoremstyle{theorem}
\newtheorem{lem}[thm]{Lemma}
\newtheorem{prop}[thm]{Proposition}
\newtheorem{cor}[thm]{Corollary}

\numberwithin{equation}{section}

\graphicspath{{./fig/}}

\title{Concordance invariant $\Upsilon$ for balanced spatial graphs using grid homology}
\author{Hajime Kubota}

\begin{document}
\maketitle
\begin{abstract}
The $\Upsilon$ invariant is a concordance invariant using knot Floer homology.
F\"{o}ldv\'{a}ri\cite{Foldvari-grid-upsilon} gives a combinatorial restructure of it using grid homology.
We extend the combinatorial $\Upsilon$ invariant for balanced spatial graphs.
Regarding links as spatial graphs, we give an upper and lower bound for the $\Upsilon$ invariant when two links are connected by a cobordism.
Also, we show that the combinatorial $\Upsilon$ invariant is a concordance invariant for knots.
\end{abstract}

\section{Introduction}
The $\tau$ invariant and the $\Upsilon$ invariant are defined by Ozsv\'{a}th, Szab\'{o} \cite{Knot-Floer-homology-and-the-four-ball-genus},\cite{Concordance-homomorphisms-from-knot-Floer-homology} using knot Floer homology.
The $\tau$ invariant and the $\Upsilon$ invariant give homomorphisms from the (smooth) knot concordance group $\mathcal{C}$ to $\mathbb{Z}$ and lower bounds for the slice genus and the unknotting number.
The $\tau$ invariant is known to prove the Milnor conjecture $g_4(T_{p,q})=\frac{1}{2}(p-1)(q-1)$ \cite{grid-tau-sarkar}.
The $\Upsilon$ invariant is a family of concordance invariants $\Upsilon_t$ defined for every $t\in[0,2]$ and the slope of the $\Upsilon$ invariant at $t=0$ equals the value of the $\tau$ invariant, so the $\Upsilon$ invariant is stronger than the $\tau$ invariant.
The $\Upsilon$ invariant shows that the subgroup of $\mathcal{C}$ generated by topologically slice knots has $\mathbb{Z}^\infty$ direct summand \cite{Concordance-homomorphisms-from-knot-Floer-homology}.

Grid homology is a combinatorial version of knot (link) Floer homology developed by Manolescu, Ozsv\'{a}th, Szab\'{o}, and Thurston \cite{on-combinatorial-link-Floer-homology}.
The original definition of knot Floer homology needs gauge theory and pseudo-holomorphic curves.
In contrast, grid homology only needs combinatorial procedures such as gazing at a planar figure called a grid diagram (figure \ref{fig:balanced}) and counting some rectangles on the grid diagram.

One of the main research directions of grid homology is to give purely combinatorial proofs for the known properties of knot Floer homology.
For example, Sarkar \cite{grid-tau-sarkar} gave a combinatorial reconstruction of the $\tau$ invariant, which we denote $\tau^{grid}$, using grid homology.
Sarkar also gave a purely combinatorial proof that the $\tau^{grid}$ is a concordance invariant.
As an application of it, a combinatorial proof of the Milnor conjecture was obtained.
F\"{o}ldv\'{a}ri \cite{Foldvari-grid-upsilon} defined the $\Upsilon^{grid}$ invariant using grid homology for $t\in[0,2]\cap\mathbb{Q}$ and evaluated the change of values on crossing changes.

A spatial graph is a smooth embedding $f\colon G\to S^3$, where $G$ is a one-dimensional CW-complex.
We assume that spatial graphs are always oriented.
A \textbf{transverse spatial graph} is a spatial graph such that for each vertex $v$, there is a small disk $D\subset S^3$ whose center is $v$ separating incoming edges and outgoing edges.
In this paper, we need one more condition: a transverse spatial graph is \textbf{balanced} if at each vertex, the number of incoming edges equals the number of outgoing edges.

Harvey, O'Donnol \cite{Heegaard_Floer_homology_of_spatial_graphs} extended grid homology to transverse spatial graphs. 
Then it is natural to explore generalizing of results of knot Floer Homology to transverse spatial graphs.
Vance \cite{grid-tau--Vance-spatial} defined the $\tau^{spatial}$ invariant for balanced spatial graphs as an extension of Sarkar's $\tau^{grid}$ invariant.
As an application of the $\tau^{spatial}$, Vance gave the bounds for the change of the $\tau^{spatial}$ invariant of two links connected by a link cobordism.

\begin{figure}[ht]
\centering
\includegraphics[scale=0.7]{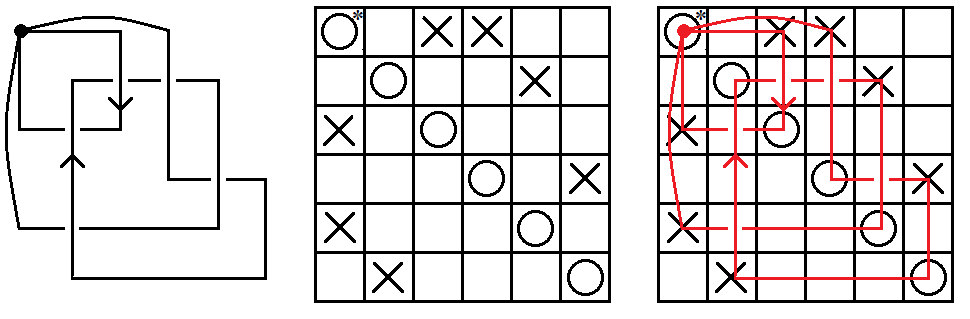}
\caption{an example of a balanced spatial graph and a graph grid diagram}
\label{fig:balanced}
\end{figure}

In this paper,  we first define $\Upsilon^{spatial}$ for balanced spatial graphs for every $t\in[0,2]$ as an extension of F\"{o}ldv\'{a}ri's knot invariant $\Upsilon^{grid}$.
Then we give a combinatorial proof that $\Upsilon^{spatial}$ (and also $\Upsilon^{grid}$) is a concordance invariant of knots.
Our $\Upsilon^{spatial}$ contains more information than F\"{o}ldv\'{a}ri's $\Upsilon^{grid}$ because $\Upsilon^{grid}$ is defined for $[0,2]\cap\mathbb{Q}$ but our $\Upsilon^{spatial}$ is defined for all $[0,2]$.

hhTo construct the $\Upsilon^{spatial}$ invariant, we use the t-modified chain complex $tCF^{-H}(g)$ from the grid chain complex in the same way as the way of Ozsv\'{a}th, Stipsicz, and Szab\'{o} \cite{Concordance-homomorphisms-from-knot-Floer-homology} and prove that the homology of the t-modified chain complex $tHF^{-H}(g)$ is (almost) independent of the choice of the graph grid diagram.
To define the $\Upsilon$ invariant using knot Floer homology, Ozsv\'{a}th, Stipsicz, and Szab\'{o} used a formal construction of the t-modified chain complex from filtered chain complex $C\mapsto C^t$, and to see the invariance, they showed that filtered chain homotopy equivalent complexes $C\simeq D$ are lifted to chain homotopy equivalent complexes $C^t\simeq D^t$.
Applying these ideas and considering the connection between the grid chain complexes and the t-modified chain complexes enable us to define $\Upsilon^{spatial}$ for all $[0,2]$ and to ensure the invariance of the $\Upsilon^{spatial}$.

\subsection{Main results.}
For $t\in[0,2]$, let $\mathcal{R}$ be a certain based ring (see Definition \ref{basedring}).
Let $W_t$ be a two-dimensional graded vector space $W_t\cong\mathbb{F}_{0}\oplus\mathbb{F}_{-1+t}$, where $\mathbb{F}=\mathbb{Z}/2\mathbb{Z}$ and their indices describe the grading (we call t-grading).
For a graded $\mathcal{R}$-module $X$, $X\llbracket a\rrbracket$ denotes a \textbf{shift} of $X$ such that $X\llbracket a\rrbracket_d=X_{d+a}$. Then,
\[
X\otimes W_t\cong X\oplus X\llbracket1-t\rrbracket.
\]

\begin{thm}
\label{main}
Let $g,g'$ be two graph grid diagrams for the same balanced spatial graph.
If their grid numbers are $n,m\ (n\geq m)$, which means that $g$ has $n\times n$ squares and $g'$ has $m\times m$ squares, then as graded $\mathcal{R}$ modules,
\[
tHF^{-H}(g)\cong tHF^{-H}(g')\otimes W_t^{\otimes (n-m)}.
\]
\end{thm}

\begin{dfn}
Let $g$ be a graph grid diagram for balanced spatial graph $f$.
For $t\in[0,1]$, let
\[
\Upsilon_g(t):=\mathrm{max}\{\mathrm{gr}_t(x)|x\in tHF^{-H}(g),x\mathrm{\ is\ homogeneous,non-torsion}\},
\]
and for $t\in[1,2]$, let $\Upsilon_g(t)=\Upsilon_g(2-t)$, where "non-torsion" means that it is non-torsion as an element of $\mathcal{R}$-module.
\end{dfn}

\begin{thm}
\label{main2}
If $g,g'$ are two graph grid diagrams for $f$,then $\Upsilon_g(t)=\Upsilon_{g'}(t)$.
\end{thm}
In other words, $\Upsilon_g(t)$ is an invariant for balanced spatial graphs. 
We will denote by $\Upsilon_f(t)$.

Let $L_1, L_2$ be two oriented links.
A \textbf{genus} $\mathbf{g}$ \textbf{link cobordism} from $L_1$ to $L_2$ is an oriented genus $g$ surface $F$ smoothly embedded in $S^3\times[0,1]$, such that $F\cap (S^3\times\{0\})=L_1$ and $F\cap (S^3\times\{1\})=-L_2$.
Two $l$-component links are \textbf{concordant} if they are connected by a cobordism consisting of $l$ disjoint annuli.

\begin{thm}
\label{main3}
Let $L_1,L_2$ be two links of $l_1,l_2$-components respectively.
If there is a genus $g$ link cobordism from $L_1$ to $L_2$, then
\[
\Upsilon_{L_1}(t)-tg-t(l_1-1)-(l_1-l_2)\leq\Upsilon_{L_2}(t)\leq\Upsilon_{L_1}(t)+tg+t(l_2-1)+(l_2-l_1).
\]
Especially, if $L_1,L_2$ are two knots $K_1,K_2$, then
\[
|\Upsilon_{K_1}(t)-\Upsilon_{K_2}(t)|\leq tg,
\]
i.e. $\Upsilon$ is a concordance invariant for knots.
\end{thm}

\begin{cor}
\label{main4}
For $t\in[0,1]$,
\[
|\Upsilon_K(t)|\leq t\cdot g_s(K),
\]
where $g_s(K)$ is the slice genus of $K$.
\end{cor}

\subsection{Some properties of $\Upsilon$}
Grid homology is a combinatorial version of knot Floer homology, so it is expected that our $\Upsilon$ satisfies the same properties as the original one of Ozsv\'{a}th, Szab\'{o} \cite{Knot-Floer-homology-and-the-four-ball-genus},\cite{Concordance-homomorphisms-from-knot-Floer-homology}.
Some properties of the original $\Upsilon$ invariant are proved using algebraic techniques, and we show some of them in the same way.

\begin{prop}
\label{prop:0}
$\Upsilon_f(0)=0$.
\end{prop}

\begin{prop}
\label{prop:1}
If two links $L_+$ and $L_-$ differ in a crossing change, then for $t\in[0,1]$,
\[
\Upsilon_{L_+}(t)\leq\Upsilon_{L_-}(t)\leq\Upsilon_{L_+}(t)+(2-t)
\]
\end{prop}
F\"{o}ldv\'{a}ri proved this property when two links are knots and $t$ is rational, by constructing concrete maps between the t-modified chain complexes (see \cite{Foldvari-grid-upsilon}). 
In this paper, we show that these maps are induced from some maps on grid chain complexes.
Note that this Proposition is weaker than the original property in \cite[Proposition 1.10]{Concordance-homomorphisms-from-knot-Floer-homology}.

For a balanced spatial graph $f$, let $\mathcal{U}(f)$ be a balanced spatial graph which is the disjoint union of a trivial knot and $f$, regarding the trivial knot as a graph consisting of one vertex and one edge.
\begin{prop}
\label{prop:2}
For a balanced spatial graph $f$,
\[
\Upsilon_{\mathcal{U}(f)}(t)=\Upsilon_f(t).
\]
\end{prop}

For a balanced spatial graph $f$ and a vertex $v$, let $f\#_v\mathcal{O}$ denote the balanced spatial graph obtained by attaching a new unknotted, unlinked edge going from $v$ to $v$.
\begin{prop}
\label{prop:3}
For a balanced spatial graph $f$,
\[
\Upsilon_{f\#_v\mathcal{O}}(t)=\Upsilon_f(t).
\]
\end{prop}

\subsection{Outline of the paper}
In Section 2, we review grid homology for transverse spatial graphs and Vance's the symmetrized Alexander filtration.
In Section 3, we define the combinatorial t-modified chain complex directly.
In Section 4, we give an alternative definition of the t-modified chain complex from the original grid chain complex using the idea of Ozsv\'{a}th, Stipsicz, and Szab\'{o} \cite{Concordance-homomorphisms-from-knot-Floer-homology} and see the two definitions are equivalent.
In Section 5, we give chain homotopy equivalences corresponding to the moves on grid diagrams.
Using these chain homotopy equivalences, in Section 6, we prove Theorem \ref{main} and Theorem \ref{main2}.
In Section 7, we consider link cobordisms on grid homology and give a proof for Theorem \ref{main3}.
Finally, in Section 8, we verify Proposition \ref{prop:0}, \ref{prop:1}, \ref{prop:2}, and \ref{prop:3}.

\section{Grid homology for transverse spatial graphs}
This section provides an overview of grid homology for transverse spatial graphs. See \cite{Heegaard_Floer_homology_of_spatial_graphs} for details. For grid homology for knots and links, see \cite{on-combinatorial-link-Floer-homology},\cite{grid-book}.

A \textbf{planar graph grid diagram} $g$ (figure \ref{fig:balanced}) is a $n\times n$ grid of squares some of which is decorated with an $X$- or $O$- (sometimes $O{}^*$-) marking such that it satisfies the following conditions.
\begin{enumerate}[(i)]
\item There is just one $O$ or $O{}^*$ on each row and column.
\item There is at least one $X$ on each row and column.
\item $O$'s (or $O{}^*$'s) and $X$'s do not share the same square.
\end{enumerate}
We denote the set of $O$-markings by $\mathbb{O}$ and the set of $X$-markings by $\mathbb{X}$.
We often use the labeling of markings as $\{O_i\}_{i=1}^n$ and $\{X_j\}_{j=1}^m$.

For any transverse spatial graph $f$, we can take graph grid diagrams representing $f$.
The relation between a spatial graph and representing grid diagram is the following: Drawing horizontal segments from the $O-$ (or $O{}^*-$) marking to the $X$-markings in each row and vertical segments from the $X$-markings to the $O-$ (or $O{}^*-$) marking in each column and assuming that the vertical segments always cross above the horizontal segments, then we can recover the spatial graph from the corresponding graph grid diagram.
We think that $O^*$-markings correspond to vertices, $O$- and $X$-markings to the interior of edges of the transverse spatial graph.

A \textbf{toroidal graph grid diagram} is a graph grid diagram that we think it as a diagram on  the torus obtained by identifying edges in a natural way.
We assume that every toroidal diagram is oriented in a natural way.

Harvey and O'Donnol showed that any two graph grid diagrams representing the same spatial graph are connected by a finite sequence of the graph grid moves.
The graph grid moves are the following three moves on the torus (See also \cite[Section 3.4]{Heegaard_Floer_homology_of_spatial_graphs})$\colon$
\begin{itemize}
\item \textbf{Cyclic permutation} (figure 2) permuting the rows or columns cyclically.
\item \textbf{Commutation$'$} (figure 3) permuting two adjacent columns satisfying the following condition; there are vertical line segments $\textrm{LS}_1,\textrm{LS}_2$ on the torus such that (1) $\mathrm{LS}_1\cup\mathrm{LS}_2$ contain all the $X$'s and $O$'s in the two adjacent columns, (2) the projection of $\mathrm{LS}_1\cup\mathrm{LS}_2$ to a single vertical circle $\beta_i$ is $\beta_i$
, and (3) the projection of their endpoints, $\partial(\mathrm{LS}_1)\cup\partial(\mathrm{LS}_2)$, to a single $\beta_i$ is precisely two points. Permuting two rows is defined in the same way.
\item \textbf{(De-)stabilization$'$} (figure 4) let $g$ be an $n\times n$ graph grid diagram and choose an $X$-marking. Then $g'$ is called a stabilization$'$ of $g$ if it is an $(n+1)\times(n+1)$ graph grid diagram obtained by adding one new row and column next to the $X$-marking of $g$, moving the $X$-marking to next column, and putting new one $O$-marking just above the $X$-marking and one $X$-marking just upper left of the $X$-marking. The inverse of stabilization$'$ is called destabilization$'$.
\end{itemize}

\begin{figure}
\centering
\includegraphics[scale=0.7]{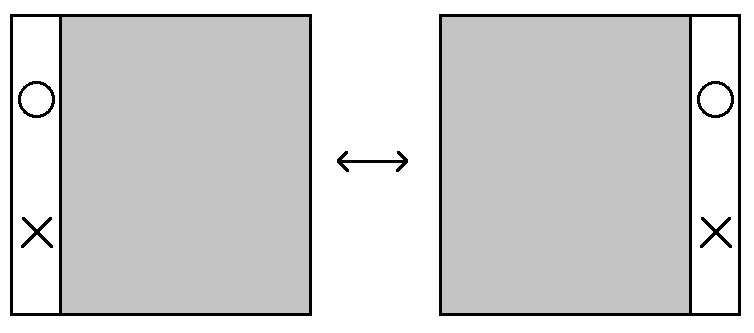}
\caption{cyclic permutation}

\centering
\includegraphics[scale=0.7]{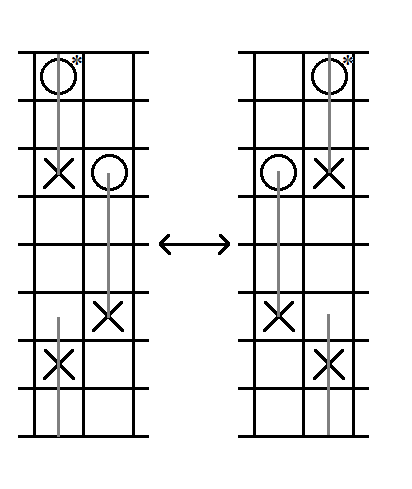}
\caption{commutation$'$, gray lines are $\mathrm{LS}_1$ and $\mathrm{LS}_2$}

\centering
\includegraphics[scale=0.7]{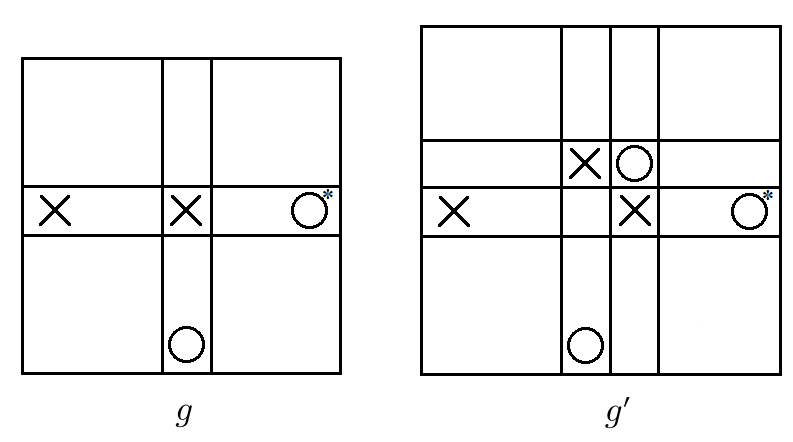}
\caption{stabilization$'$}
\end{figure}

We write the horizontal circles and vertical circles which separate the torus into squares as $\boldsymbol{\alpha}=\{\alpha_i\}_{i=1}^n$ and $\boldsymbol{\beta}=\{\beta_j\}_{j=1}^n$ respectively.

A \textbf{state} $\mathbf{x}$ of $g$ is a bijection $\boldsymbol{\alpha}\rightarrow\boldsymbol{\beta}$.
We denote by $\mathbf{S}(g)$ the set of states of $g$.
We describe a state as $n$ points on a toroidal graph grid diagram (figure \ref{fig:state}).

\begin{figure}[ht]
\centering
\includegraphics[scale=0.6]{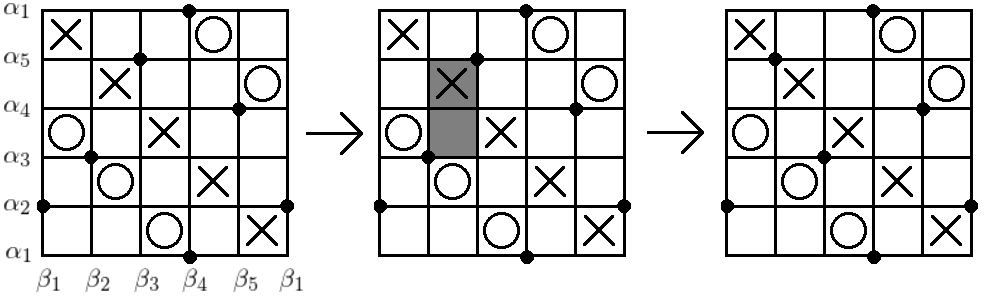}
\caption{an example of a state and a rectangle}
\label{fig:state}
\end{figure}

We think that there are $n\times n$ squares on $g$ that is separated by $\boldsymbol{\alpha}\cup\boldsymbol{\beta}$.
Fix $\mathbf{x,y}\in S$(g), a \textbf{domain} $p$ from $\mathbf{x}$ to $\mathbf{y}$ is a formal sum of the closure of squares satisfying $\partial(\partial_\alpha p)=\mathbf{y}-\mathbf{x}$ and $\partial(\partial_\beta p)=\mathbf{x}-\mathbf{y}$, where $\partial_\alpha p$ is the portion of the boundary of $p$ in the horizontal circles $\alpha_1\cup\dots\cup\alpha_n$ and $\partial_\beta p$ is the portion of the boundary of $p$ in the vertical ones.
A domain $p$ is \textbf{positive} if the coefficient of any square is non-negative.
Here we always assume that any domain is positive.
Let $\pi(\mathbf{x,y})$ denote the set of domains from $\mathbf{x}$ to $\mathbf{y}$.

Consider $\mathbf{x,y\in S}(g)$ that coincide with $n-2$ points.
An \textbf{rectangle} $r$ from $\mathbf{x}$ to $\mathbf{y}$ is a domain that satisfies that $\partial r$ is the union of four segments.
A rectangle $r$ is an \textbf{empty rectangle} if $\mathbf{x}\cap\mathrm{Int}(r)=\mathbf{y}\cap\mathrm{Int}(r)=\emptyset$.
Let $\mathrm{Rect}^\circ(\mathbf{x,y})$ be the set of empty rectangles from $\mathbf{x}$ to $\mathbf{y}$.

The grid chain complex $(CF^-(g),\partial^-)$ is a module over $\mathbb{F}[U_1,\dots,U_n]$ freely generated by $\mathbf{S}(g)$, where $\mathbb{F}=\mathbb{Z}/2\mathbb{Z}$ and the $U_i$'s are the formal variables corresponding to the $O_i$'s in $g$.
The differential $\partial^-$ is defined as counting empty rectangles by
\[
\partial^-(\mathbf{x})=\sum_{\mathbf{y}\in\mathbf{S}(g)}\left(
\sum_{r\in \mathrm{Rect}^\circ(\mathbf{x,y})}
U_1^{O_1(r)}\cdots U_n^{O_n(r)}
\right)\mathbf{y},
\]
where $O_i(r)=1$ if $r$ contains $O_i$ and $O_i(r)=0$ otherwise for $i=1,\dots,n$. 

There are two gradings for $CF^-(g)$, the \textbf{Maslov grading} and the \textbf{Alexander grading}.
A planar realization of toroidal diagram $g$ is a planar figure obtained by cutting toroidal diagram $g$ along $\alpha_i$ and $\beta_j$ for some $i$ and $j$, and putting on $[0,n)\times[0,n)\in\mathbb{R}^2$ in a natural way.
For two points $(a_1,a_2),(b_1,b_2)\subset\mathbb{R}^2$, let $(a_1,a_2)<(b_1,b_2)$ if $a_1<b_1$ and $a_2<b_2$.
For two sets of finitely points $A,B\subset\mathbb{R}^2$, let $\mathcal{I}(A,B)$ be the number of pairs $a\in A,b\in B$ with $a<b$ and let $\mathcal{J}(A,B)=(\mathcal{I}(A,B)+\mathcal{I}(B,A))/2$.
Let $m_i$ be the number of $X$-markings in the row caontaining $O_i$ ($i=1,\dots,n$).
Then for $\mathbf{x\in S}(g)$, the Maslov grading $M(\mathbf{x})$ and the Alexander grading $A(\mathbf{x})$ are defined by
\begin{align}
M(\mathbf{x})=\mathcal{J}(\mathbf{x}-\mathbb{O},\mathbf{x}-\mathbb{O})+1,
\label{mm}
\\
A(\mathbf{x})=\mathcal{J}(\mathbf{x},\mathbb{X}-\sum_{i=1}^nm_iO_i).
\label{aa}
\end{align}
These two gradings are extended to the whole of $CF^-(g)$ by
\begin{align}
\label{uu}
M(U_i)=-2,A(U_i)=-m_i\ (i=1,\dots,n).
\end{align}
Note that the Alexander grading here is the same as Vance's definition, while the original Alexander grading by Harvey and O'Donnol has values in $H_1(S^3-f(G))$, where $f\colon G\rightarrow S^3$ is the transverse spatial graph represented by $g$. The Alexander grading here is given from the original one by taking the canonical homomorphism $H_1(S^3-f(G))\rightarrow\mathbb{Z}$ which sends the generators to 1.
Also, note that the Alexander grading is not well-defined as a toroidal diagram, however, relative Alexander grading $A^{rel}(\mathbf{x,y})=A(\mathbf{x})-A(\mathbf{y})$ is well-defined.

It is shown that the differential $\partial^-$ drops Maslov grading by 1 and preserves or drops Alexander grading, so $(CF^-(g),\partial^-)$ is a Maslov graded, Alexander filtered chain complex \cite[Proposition 3.7]{grid-tau--Vance-spatial}.

Suppose the $O$-markings are labeled so that $O_1,\dots,O_V$ are $O^*$-markings and $O_{V+1},\dots,O_n$ are $O$-markings. Let $\mathcal{U}$ be the minimal subcomplex of $CF^-(g)$ containing $U_1CF^-(g)\cup\dots\cup U_VCF^-(g)$. Then $(\widehat{CF}(g),\widehat{\partial})$ is a Maslov graded, Alexander filtered chain complex over $\mathbb{F}$-vector space obtained by letting $\widehat{CF}(g)=CF^-(g)/\mathcal{U}$ and $\widehat{\partial}$ be the map induced by $\partial^-$.

We denote by $\{\mathcal{F}^-_m(g)\}_{m\in\mathbb{Z}}$ (respectively $\{\widehat{\mathcal{F}}_m(g)\}_{m\in\mathbb{Z}}$) the Alexander filtration of $CF^-(g)$ (respectively $\widehat{CF}(g)$).

\begin{dfn}[{\cite[Definition3.10]{grid-tau--Vance-spatial}}]
\label{symmetrizedA}
For a graph grid diagram $g$, define the symmetrized Alexander filtration $\{\widehat{\mathcal{F}}_m(g)\}_{m\in\frac{1}{2}\mathbb{Z}}$ to be the absolute Alexander filtration obtained by fixing the relative Alexander grading so that $m_{\mathrm{max}}(g)=-m_{\mathrm{min}}(g)$, where
\begin{align*}
m_{\mathrm{max}}(g)&=\mathrm{max}\left\{m|H_*(\widehat{\mathcal{F}}_m(g)/\widehat{\mathcal{F}}_{m-1}(g))\neq0\right\},\\
m_{\mathrm{min}}(g)&=\mathrm{min}\left\{m|H_*(\widehat{\mathcal{F}}_m(g)/\widehat{\mathcal{F}}_{m-1}(g))\neq0\right\}.
\end{align*}
\end{dfn}
\begin{dfn}
The \textbf{symmetrized Alexander grading} $A^H\colon\mathbf{S}(g)\rightarrow\frac{1}{2}\mathbb{Z}$ is determined by symmetrized Alexander filtration $\{\widehat{\mathcal{F}}_m(g)\}_{m\in\frac{1}{2}\mathbb{Z}}$ so that for $\mathbf{x\in S}(g)$, the value of $A^H(\mathbf{x})$ is the maximal filtration level which $\mathbf{x}\in\widehat{CF}(g)$ belongs.
\end{dfn}
Let $CF^{-H}(g)$ be a Maslov graded, symmetrized Alexander filtered chain complex obtained from $CF^-(g)$ by using the symmetrized Alexander grading $A^H$ rather than the Alexander grading $A$.

The homology of the associated graded object of $CF^{-H}(g)$ is an invariant for balanced spatial graphs \cite[Theorem 3.15]{grid-tau--Vance-spatial}.

\section{The t-modified chain complex}
This section provides the definition of the t-modified chain complex $tCF^{-H}(g)$.
The t-modified chain complex here is the extension of one of F\"{o}ldv\'{a}ri.

First of all, we define the based ring $\mathcal{R}$ of $tCF^{-H}(g)$.

A set $A\subset\mathbb{R}$ is \textbf{well-ordered} if any subset $A'\subset A$ has a minimal element.
\begin{dfn}[{\cite[Definition 3.1]{Concordance-homomorphisms-from-knot-Floer-homology}}]
\label{basedring}
Let $\mathbb{R}_{\geq0}$ denote the set of nonnegative real numbers. The ring of long power series $\mathcal{R}$ is defined as follows.
As an abelian group, $\mathcal{R}$ is the group of formal sums
\[
\left\{\sum_{\alpha\in A}v^\alpha|A\subset\mathbb{R}_{\geq0},A\ \mathrm{is\ well-ordered}\right\},
\]
\end{dfn}

The sum in $\mathcal{R}$ is given by the formula
\[
\left(\sum_{\alpha\in A}v^\alpha\right)+\left(\sum_{\beta\in B}v^\beta\right)=\sum_{\gamma\in C=A\cup B\setminus A\cap B}v^\gamma\ ,
\]
and the product is given by the formula
\[
\left(\sum_{\alpha\in A}v^\alpha\right)\cdot\left(\sum_{\beta\in B}v^\beta\right)=\sum_{\gamma\in A+B}\#\{(\alpha,\beta)\in A\times B|\alpha+\beta=\gamma\}\cdot v^\gamma,
\]
where
\[
A+B=\left\{\gamma|\gamma=\alpha+\beta\ for\ some\ \alpha\in A\ and\ \beta\in B\right\}.
\]
It is straightforward to check that the above operations are well-defined.

Then we define the t-modified chain complex $tCF^{-H}(g)$.
For a domain $p$ as a formal sum of squares, let $O_i(p)$ denote the coefficient of the square containing $O_i$, and let
\[
|\mathbb{O}\cap p|:=\sum_{i=1}^nO_i(p).
\]
We define $|\mathbb{X}\cap p|$ in the same manner.

\begin{dfn}
\label{t-mod}
For $t\in[0,2]$, the t-modified grid complex $tCF^{-H}(g)$ is a free module over $\mathcal{R}$ generated by $\mathbf{S}(g)$ with the $\mathcal{R}$-module endmorphism $\partial^-_t$ defined by
\begin{equation}
\label{differential_combinatorial}
\partial_t^-(\mathbf{x})=\sum_{\mathbf{y}\in\mathbf{S}(g)}\left(
\sum_{r\in \mathrm{Rect}^\circ(\mathbf{x,y})}v^{t|\mathbb{X}\cap r|+2|\mathbb{O}\cap r|-t\left(\sum_{O_i\in\mathbb{O}\cap r}m_i\right)}
\right)\mathbf{y},
\end{equation}
where $m_i$ is the number of $X$-markings on the same row as $O_i$ ($i=1,\dots,n$).
\end{dfn}

\begin{dfn}
For $\mathbf{x\in S}(g)$ and $\alpha>0$, the \textbf{t-grading} $\mathrm{gr}_t$ is,
\begin{align}
\mathrm{gr}_t(v^\alpha\mathbf{x})=M(\mathbf{x})-tA^H(\mathbf{x})-\alpha,
\end{align}
where $A^H$ is the symmetrized Alexander grading.
\end{dfn}

\begin{prop}
\label{t-ch1}
$\partial_t^-\circ\partial_t^-=0$.
\end{prop}
\begin{proof}
For states $\mathbf{x}$ and $\mathbf{z}$ and a fixed domain $p\in\pi(\mathbf{x,z})$ denote by $N(p)$ the number of ways to decompose $p$ as a composite of two empty rectangles $r_1*r_2$. Note that if $p=r_1*r_2$ for some $r_1\in\mathrm{Rect}^\circ(\mathbf{x,y}),r_2\in\mathrm{Rect}^\circ(\mathbf{y,z})$, then
\begin{align*}
|\mathbb{X}\cap p|&=|\mathbb{X}\cap r_1|+|\mathbb{X}\cap r_2|,\\|\mathbb{O}\cap p|&=|\mathbb{O}\cap r_1|+|\mathbb{O}\cap r_2|,\\\sum_{O_i\in\mathbb{O}\cap p}m_i&=\sum_{O_i\in\mathbb{O}\cap r_1}m_i+\sum_{O_i\in\mathbb{O}\cap r_2}m_i.
\end{align*}
It follows that for $\mathbf{x\in S}(g)$,
\begin{equation}
\label{pp}
\partial_t^-\circ\partial_t^-(\mathbf{x})=\sum_{\mathbf{z\in S}(g)}\left(
\sum_{p\in\pi(\mathbf{x,z})}N(p)v^{t|\mathbb{X}\cap p|+2|\mathbb{O}\cap p|-t\left(\sum_{O_i\in\mathbb{O}\cap p}m_i\right)}
\right)\mathbf{z}.
\end{equation}
If $\#\{\mathbf{x\setminus(x\cap z)}\}=4$ or $\#\{\mathbf{x\setminus(x\cap z)}\}=3$, the same argument in \cite[Theorem 3.2]{Concordance-homomorphisms-from-knot-Floer-homology} shows that $N(p)$ is even.
If $\mathbf{x=z}$, $p$ is an annulus and $N(p)=1$.
Since $r_1$ and $r_2$ are empty, this annulus has a height or width equal to 1.
Such an annulus is called a thin annulus.
For each $\mathbf{x}$, there are $2n$ thin annuli appearing in (\ref{pp}).
We can pair annuli that contain the same $O$-marking. 
If $(p_1,p_2)$ is such a pair, then $|p_1\cap\mathbb{X}|=|p_2\cap\mathbb{X}|$ because $f$ is a balanced spatial graph. 
So all terms are canceled in pairs (we are working modulo 2).
\end{proof}

\begin{prop}
\label{t-ch2}
The map $\partial^-_t$ drops the t-grading by one.
\end{prop}
\begin{proof}
By \cite[Lemma 2.5]{on-combinatorial-link-Floer-homology} and \cite[Lemma 3.5]{grid-tau--Vance-spatial}, for $r\in\mathrm{Rect}^\circ(\mathbf{x,y})$,
\begin{align}
\label{m-m}
M(\mathbf{x})-M(\mathbf{y}) &=1-2|r\cap\mathbb{O}|,\\
\label{a-a}
A^H(\mathbf{x})-A^H(\mathbf{y}) &=|r\cap\mathbb{X}|-\sum_{O_i\in r\cap\mathbb{O}}m_i.
\end{align}
If $v^{t|\mathbb{X}\cap r|+2|\mathbb{O}\cap r|-t\left(\sum_{O_i\in\mathbb{O}\cap r}m_i\right)}
\mathbf{y}$ appears in $\partial^-_t(\mathbf{x})$, then
\begin{align*}
\begin{split}
&\mathrm{gr}_t(v^{t|\mathbb{X}\cap r|+2|\mathbb{O}\cap r|-t\left(\sum_{O_i\in\mathbb{O}\cap r}m_i\right)}
\mathbf{y})\\ 
&=M(\mathbf{y})-tA^H(\mathbf{y})-\left(t|\mathbb{X}\cap r|+2|\mathbb{O}\cap r|-t\left(\sum_{O_i\in\mathbb{O}\cap r}m_i\right)\right)\\
&=M(\mathbf{x})-1+2|r\cap\mathbb{O}|-t\left(A^H(\mathbf{x})-|r\cap\mathbb{X}|-\left(\sum_{O_i\in\mathbb{O}}m_i\right)\right)\\&\quad-\left(t|\mathbb{X}\cap r|+2|\mathbb{O}\cap r|-t\left(\sum_{O_i\in\mathbb{O}\cap r}m_i\right)\right)
\\
&=M(\mathbf{x})-tA^H(\mathbf{x})-1.
\end{split}
\end{align*}
\end{proof}

\begin{prop}
\label{t-ch}
$(tCF^{-H}(g),\partial^-_t)$ is a t-graded chain complex over $\mathcal{R}$.
\end{prop}
\begin{proof}
From Proposition \ref{t-ch1} and Proposition \ref{t-ch2}, we conclude that $(tCF^{-H}(g),\partial^-_t)$ is a t-graded chain complex.
\end{proof}

\begin{dfn}
Let $tCF^{-H}_{d}(g)=\{\alpha\in tCF^-(g)|\mathrm{gr}_t(\alpha)=d\}$, and we define \textbf{t-modified graph grid homology} of $g$ to be
\begin{align*}
tHF^{-H}_{d}(g)&=\frac{\mathrm{Ker}(\partial_t^-)\cap tCF^{-H}_{d}(g)}{\mathrm{Im}(\partial_t^-)\cap tCF^{-H}_{d}(g)},\\
tHF^{-H}(g)&=\bigoplus_{d}tHF^{-H}_{d}(g).
\end{align*}
\end{dfn}

\section{Formal construction of the t-modified chain complex}
In this section, we give an alternative definition of the t-modified chain complex using the grid chain complex.
See \cite[Section 4]{Concordance-homomorphisms-from-knot-Floer-homology} for details. 

Suppose that $C$ is a finitely generated, Maslov graded, Alexander filtered chain complex over $\mathbb{F}[U]$. Let $\mathbf{x}$ be a generator of $C$ over $\mathbb{F}[U]$, with Maslov grading $M(\mathbf{x})$.
Since multiplication by $U$ drops the Maslov grading by 2, elements of Maslov grading $M(\mathbf{x})-1$ are linear combinations of elements of the form $U^{\frac{M(\mathbf{y})-M(\mathbf{x})+1}{2}}\mathbf{y}$, where $\mathbf{y}$ is a generator.
Then the differential on $C$ can be written as
\begin{equation}
\label{dd}
\partial(\mathbf{x})=\sum_{\mathbf{y}}c_{\mathbf{x},\mathbf{y}}\cdot U^{\frac{M(\mathbf{y})-M(\mathbf{x})+1}{2}}\mathbf{y},
\end{equation}
where $c_{\mathbf{x},\mathbf{y}}\in\{0,1\}$.

\begin{dfn}[Definition 4.1,\cite{Concordance-homomorphisms-from-knot-Floer-homology}]
\label{t-mod2}
For $t\in[0,2]$, suppose that $C$ is a finitely generated, Maslov graded, Alexander filtered chain complex over $\mathbb{F}[U]$, and let $\mathcal{R}$ be the ring of Definition \ref{basedring} (containing $\mathbb{F}[U]$ by $U=v^2$).
\textbf{The formal t-modified chain complex} $C^t$ of $C$ is defined as follows$\colon$
\begin{itemize}
\item As an $\mathcal{R}$-module, $C^t=C\otimes_{\mathbb{F}[U]}\mathcal{R}$
\item For each generator $\mathbf{x}$ of $C$, define $\mathrm{gr}_t(v^\alpha\mathbf{x})=M(\mathbf{x})-tA(\mathbf{x})-\alpha$
\item Endow the graded module $C^t$ with a differential
\begin{equation}
\label{differential_formal}
\partial_t(\mathbf{x})=\sum_{\mathbf{y}}c_{\mathbf{x},\mathbf{y}}\cdot v^{\mathrm{gr}_t(\mathbf{y})-\mathrm{gr}_t(\mathbf{x})+1}\mathbf{y},
\end{equation}

where $c_{\mathbf{x},\mathbf{y}}$ are determined by (\ref{dd}).
\end{itemize}
\end{dfn}

\begin{dfn}
For a graph grid diagram $g$, let $CF^{-H}_U(g)$ denote the quotient chain complex $CF^{-H}_U(g)=\frac{CF^{-H}(g)}{U_1=\dots=U_n}$ as a Maslov graded chain complex over $\mathbb{F}[U]$-module.
\end{dfn}
Note that $CF^{-H}_U(g)$ is a Maslov graded chain complex but \textbf{not} an Alexander filtered chain complex because the drop in the Alexander grading of each $U_i$ differs by (\ref{uu}).
However, Definition \ref{t-mod2} works for $CF^{-H}_U(g)$ because each generator has its Alexander grading.
Therefore we can define the formal t-modified chain complex $(CF^{-H}_U(g))^t$.

\begin{prop}
\label{U=t}
$(CF^{-H}_U(g))^t$ (applying Definition \ref{t-mod2}) is isomorphic to $tCF^{-H}(g)$ (from Definition \ref{t-mod}) as graded chain complexes.
\end{prop}

\begin{proof}
Identifying the generators and their gradings is natural, so we only need to check that (\ref{differential_combinatorial}) and (\ref{differential_formal}) are equivalent.
In other words, we will check that for $\mathbf{x,y\in S}(g)$, 
\begin{equation}
c_{\mathbf{x,y}}=1\Leftrightarrow \#(\mathrm{Rect}^\circ(\mathbf{x,y}))=1,
\label{ct=ct1}
\end{equation}
and
\begin{equation}
\label{ct=ct2}
\mathrm{gr}_t(\mathbf{y})-\mathrm{gr}_t(\mathbf{x})+1=t|\mathbb{X}\cap r|+2|\mathbb{O}\cap r|-t(\sum_{O_i\in\mathbb{O}\cap r}m_i).
\end{equation}
We will consider the case of $\#(\mathrm{Rect}^\circ(\mathbf{x,y}))=2$.
Let $r,r'$ be the two rectangles from $\mathbf{x}$ to $\mathbf{y}$.
Then $U_1^{O_1(r)}\cdots U_n^{O_n(r)}\mathbf{y}$ and $U_1^{O_1(r')}\cdots U_n^{O_n(r')}\mathbf{y}$ appear in $\partial^-(\mathbf{x})$.
Because the numbers of $O$- and $O^*$-markings the two rectangles contain are the same, the quotient map $CF^{-H}(g)\to CF^{-H}_U(g)$ sends these two terms to the same element (that is $U^{\frac{M(\mathbf{y})-M(\mathbf{x})+1}{2}}\mathbf{y}$), so they are canceled modulo 2.
Therefore we get $c_{\mathbf{x,y}}=0$.

Also it is clear that $\#(\mathrm{Rect}^\circ(\mathbf{x,y}))=0\Rightarrow c_{\mathbf{x,y}}=0$ and that $\#(\mathrm{Rect}^\circ(\mathbf{x,y}))=1\Rightarrow c_{\mathbf{x,y}}=1$.
Therefore (\ref{ct=ct1}) is proved.

We have (\ref{ct=ct2}) directly from (\ref{m-m}) and (\ref{a-a}).
\end{proof}

Next, we introduce the following proposition playing an important role in this paper.

\begin{prop}[{\cite[Proposition 4.4]{Concordance-homomorphisms-from-knot-Floer-homology}}]
\label{ct=ct}
Let $f\colon C\rightarrow C'$ be a Maslov graded, Alexander filtered chain map between chain complexes over $\mathbb{F}[U]$. There is a corresponding graded chain map $f^t\colon C^t\rightarrow (C')^t$, with the following properties$\colon$
\begin{itemize}
\item If $f\colon C\rightarrow C'$ and $g\colon C'\rightarrow C''$ are two Maslov graded, Alexander filtered chain maps, then
\[
(g\circ f)^t=g^t\circ f^t
\]
\item If $f,g\colon C\rightarrow C'$ are chain homotopic to each other, then $f^t$ and $g^t$ are chain homotopic to each other. In particular, filtered chain homotopy equivalent complexes are transformed by the construction $C\mapsto C^t$ into homotopy equivalent complexes.
\end{itemize}
\end{prop}

Again note that this proposition can be applied as $CF^{-H}_U(g)\mapsto(CF^{-H}_U(g))^t$ even if $CF^{-H}_U(g)$ is just Maslov graded chain complex with $A^H$ grading only for generators.
Similar to (\ref{dd}), Maslov graded chain map $f\colon CF^{-H}_U(g)\rightarrow CF^{-H}_U(g')$ can be written as
\[
f(\mathbf{x})=\sum_{\mathbf{y}}c_{\mathbf{x},\mathbf{y}}\cdot U^{\frac{M(\mathbf{y})-M(\mathbf{x})}{2}}\mathbf{y},
\]
then the corresponding graded chain map is
\[
f^t(\mathbf{x})=\sum_{\mathbf{y}}c_{\mathbf{x},\mathbf{y}}\cdot v^{\mathrm{gr}_t(\mathbf{y})-\mathrm{gr}_t(\mathbf{x})}\mathbf{y}.
\]
This construction gives induced chain homotopy equivalence for the t-modified chain complexes.

\section{Chain homotopy equivalences for grid chain complexes}
This section provides the filtered chain homotopy equivalences for grid chain complexes connected by each graph grid move.
These filtered chain homotopy equivalences are lifted into chain homotopy equivalences for the t-modified chain complexes by applying Proposition \ref{ct=ct}.
Chain homotopy equivalences for a cyclic permutation and a commutation$'$ are already known.
We mainly observe the case of stabilization$'$.

\subsection{cyclic permutation}
\begin{prop}
\label{chainho1}
If $g$ and $g'$ are connected by a single8 cyclic permutation, then as Maslov graded, Alexander filtered chain complexes, $CF^{-H}(g)$ and $CF^{-H}(g')$ are chain homotopy equivalent.

\end{prop}
\begin{proof}
There is a natural bijection $c\colon\mathbf{S}(g)\rightarrow\mathbf{S}(g')$. By \cite[Section 3.3 and Theorem 3.15]{grid-tau--Vance-spatial}, $c$ induces an isomorphism of Maslov graded, symmetrized Alexander filtered chain complexes $c\colon CF^{-H}(g)\rightarrow CF^{-H}(g')$.
As $c\colon\mathbf{S}(g)\rightarrow\mathbf{S}(g')$ is a bijection, there is an isomorphism $c^{-1}\colon CF^{-H}(g')\rightarrow CF^{-H}(g)$ induced by $c^{-1}$ such that $c^{-1}\circ c=\mathrm{id}$ and $c\circ c^{-1}=\mathrm{id}$.
\end{proof}

\subsection{commutation$'$}
\begin{prop}
\label{chainho2}
If $g$ and $g'$ are connected by a single commutation$'$, then as Maslov graded, Alexander filtered chain complexes, $CF^{-H}(g)$ and $CF^{-H}(g')$ are chain homotopy equivalent.
\end{prop}
\begin{proof}
By \cite[Lemma 3.21]{grid-tau--Vance-spatial} and \cite[Proposition3.2]{on-combinatorial-link-Floer-homology}, there are chain homotopic equivalences $p'_{\beta\gamma}\colon CF^{-}(g)\rightarrow CF^{-}(g')$ and $p'_{\gamma\beta}\colon CF^{-}(g')\rightarrow CF^{-}(g)$. Since \cite[Theorem 3.15]{grid-tau--Vance-spatial}, they preserve the symmetrized Alexander filtration.
\end{proof}

\subsection{stabilization$'$}
\label{sub:sta}
Assume that $g'$ is obtained from $g$ by a single stabilization$'$ and that $CF^{-H}(g)$ is a chain complex over $\mathbb{F}[U_2,\dots,U_n]$ and $CF^{-H}(g')$ is a chain complex over $\mathbb{F}[U_1,\dots,U_n]$.
Number the $O$-markings of $g'$ so that $O_1$ is the new one and $O_2$ is the $O$-marking in the row just below $O_1$.
We will also assume that $X_1$ lies in the same row as $O_1$ and $X_2$ is the $X$-marking just below $O_1$.
We denote by $c$ the intersection point of the new horizontal and vertical circles in $g'$ (Figure \ref{fig:sta2}).
Note that there may be more $X$-markings in the row that includes $O_2$ (in which case $O_2$ is an $O^*$-marking).

\begin{figure}
\centering
\includegraphics[scale=0.7]{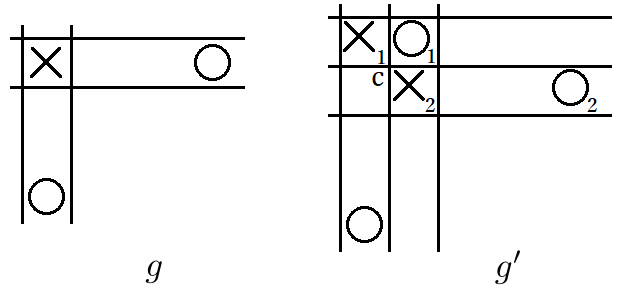}
\caption{labeling of stabilized graph grid diagram}
\label{fig:sta2}
\end{figure}

For a Maslov graded, Alexander filtered chain complex $X$, let $X'=X\llbracket a,b\rrbracket$ denote the shifted complex such that $\mathcal{F}_sX'_d=\mathcal{F}_{s+b}X_{d+a}$, where $\mathcal{F}_{s}X_{d}$ is the submodule of $X$ whose grading is $d$ and filtration level is $s$.

Considering the point $c$, we will decompose the set of states $\mathbf{S}(g')$ as the disjoint union $\mathbf{I}(g')\cup\mathbf{N}(g')$, where $\mathbf{I}(g')=\{\mathbf{x\in S}(g')|c\in\mathbf{x}\}$ and $\mathbf{N}(g')=\{\mathbf{x\in S}(g')|c\notin\mathbf{x}\}$.
This decomposition gives a decomposition of $tCF^-(g')=I\oplus N$ as an $\mathcal{R}$-module, where $I$ and $N$ are the spans of $\mathbf{I}(g')$ and $\mathbf{N}(g')$ respectively.

There is a natural bijection between $\mathbf{I}(g')$ and $\mathbf{S}(g)$, given by
\[
e\colon\mathbf{I}(g')\to\mathbf{S}(g),\ \mathbf{x}\cup\{c\}\mapsto\mathbf{x}.
\]

Let $(CF^{-H}(g)[U_1],\partial)$ be a chain complex defined by $CF^{-H}(g)[U_1]=CF^{-H}(g)\otimes_{\mathbb{F}[U_2,\dots,U_n]}\mathbb{F}[U_1,\dots,U_n]$ with the differential $\partial=\partial^-\otimes\mathrm{id}$.
Consider a chain complex $\mathrm{Cone}(U_1-U_2\colon CF^{-H}(g)[U_1]\to CF^{-H}(g)[U_1])$.

For the stabilization$'$ invariance, we will prove the following proposition.
\begin{prop}
\label{chainho3}
$CF^{-H}(g')$ and $\mathrm{Cone}(U_1-U_2\colon CF^{-H}(g)[U_1]\to CF^{-H}(g)[U_1])$ are chain homotopy equivalent complexes.
\end{prop}
This can be proved in the same way as the proof of stabilization invariance of grid homology for knots and links written in \cite[Sections 13.3.2 and 13.4.2]{grid-book}.
Because we are working on Alexander filtered version, the domains which appear in the chain maps can contain any $X$-markings of $\mathbb{X}$, in other words, they are independent of the $X$-markings.
Thus we can use the same domains as in \cite[Definitions 13.3.7 and 13.4.2]{grid-book} to show this Proposition.
We construct chain homotopy equivalence using the filtered destabilization map $\mathcal{D}$ (\cite[Definition 13.3.10]{grid-book}), the filtered stabilization map $\mathcal{S}$ (\cite[Definition 13.4.3]{grid-book}) and chain homotopy $\mathcal{K}$ (Definition \ref{dfn:k}).
Note that these maps are the answer to \cite[Exercise 13.4.4]{grid-book}.

\begin{dfn}[{\cite[Definition 13.3.7]{grid-book}}]
For $\mathbf{x\in S}(g')$ and $\mathbf{y\in I}(g')$, a positive domain $p\in\pi(\mathbf{x,y})$ is said to be of type $\bm{iL}$ if it is trivial or it satisfies the following conditions$\colon$
\begin{itemize}
\item At each corner point in $\mathbf{x\cup y}\setminus\{c\}$, at least three of the four adjoining squares have the local multiplicity zero.
\item Three of four squares attaching at the corner point $c$ have the local multiplicity $k$ and at the southwest square meeting $c$ the local multiplicity is $k-1$.
\item $\mathbf{y}$ has $2k+1$ components that are not in $\mathbf{x}$.
\end{itemize}
And, a positive domain $p\in\pi(\mathbf{x,y})$ is said to be of type $\bm{iR}$ if it is trivial or it satisfies the following conditions$\colon$
\begin{itemize}
\item At each corner point in $\mathbf{x\cup y}\setminus\{c\}$, at least three of the four adjoining squares have the local multiplicity zero.
\item Three of four squares attaching at the corner point $c$ have the local multiplicity $k$ and at the southeast square meeting $c$ the local multiplicity is $k+1$.
\item $\mathbf{y}$ has $2k+2$ components that are not in $\mathbf{x}$.
\end{itemize}
The set of domains of type $iL$ (respectively type $iR$) from $\mathbf{x}$ to $\mathbf{y}$ is denoted $\pi^{iL}(\mathbf{x,y})$ (respectively $\pi^{iR}(\mathbf{x,y})$).
\end{dfn}

\begin{figure}
\centering
\includegraphics[scale=0.48]{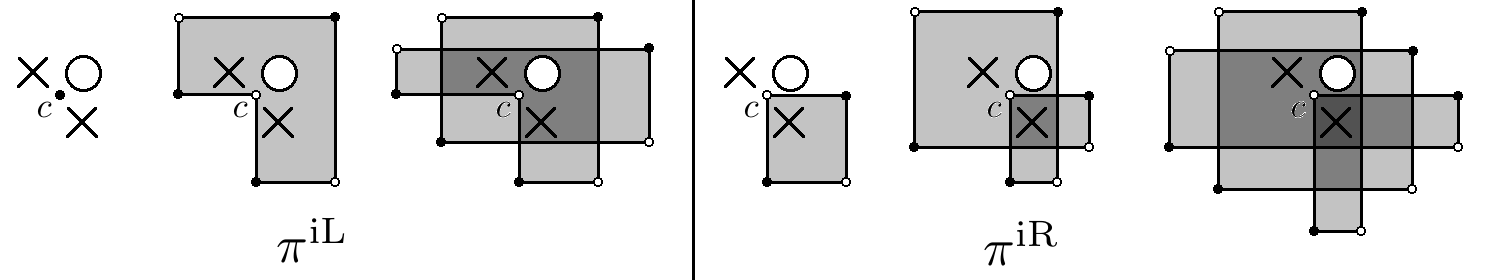}
\caption{Examples of $\pi^{iL}(\mathbf{x,y})$ and $\pi^{iR}(\mathbf{x,y})$}
\label{fig:ilir}
\includegraphics[scale=0.48]{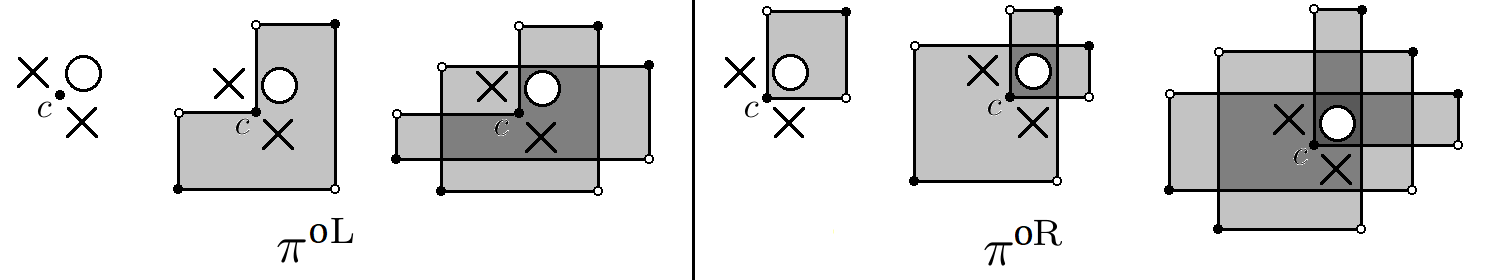}
\caption{Examples of $\pi^{oL}(\mathbf{x,y})$ and $\pi^{oR}(\mathbf{x,y})$}
\label{fig:olor}
\end{figure}

\begin{dfn}[{\cite[Definition 13.4.2]{grid-book}}]
For $\mathbf{x\in S}(g')$ and $\mathbf{y\in I}(g')$, a positive domain $p\in\pi(\mathbf{x,y})$ is said to be of type $\bm{oL}$ if it is trivial or it satisfies the following conditions$\colon$
\begin{itemize}
\item At each corner point in $\mathbf{x\cup y}\setminus\{c\}$, at least three of the four adjoining squares have local multiplicity zero.
\item Three of four squares attaching at the corner point $c$ have the local multiplicity $k$ and at the northwest square meeting $c$ the local multiplicity is $k-1$.
\item $\mathbf{y}$ has $2k+1$ components that are not in $\mathbf{x}$.
\end{itemize}
And, a positive domain $p\in\pi(\mathbf{x,y})$ is said to be of type $\bm{oR}$ if it is trivial or it satisfies the following conditions$\colon$
\begin{itemize}
\item At each corner point in $\mathbf{x\cup y}\setminus\{c\}$, at least three of the four adjoining squares have local multiplicity zero.
\item Three of four squares attaching at the corner point $c$ have the local multiplicity $k$ and at the northeast square meeting $c$ the local multiplicity is $k+1$.
\item $\mathbf{y}$ has $2k+2$ components that are not in $\mathbf{x}$.
\end{itemize}
The set of domains of type $oL$ and $oR$ from $\mathbf{x}$ to $\mathbf{y}$ is denoted $\pi^{oL}(\mathbf{x,y})$ and $\pi^{oR}(\mathbf{x,y})$ respectively.
\end{dfn}

\begin{figure}
\centering
\includegraphics[scale=0.5]{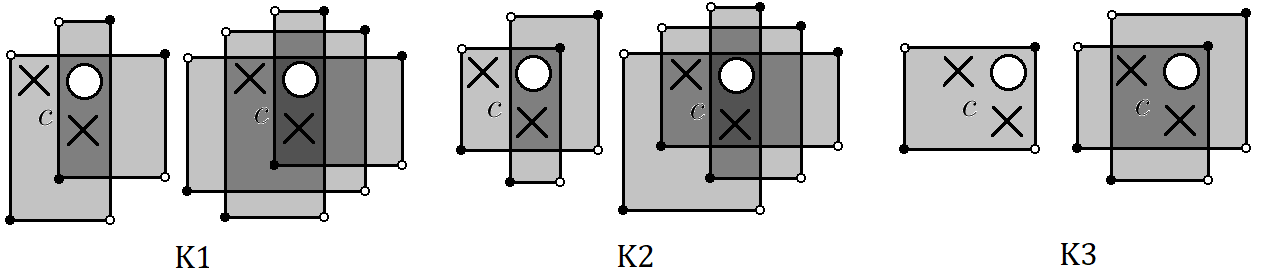}
\caption{Examples of $\pi^{K}$ }
\label{fig:typek}
\end{figure}

Let $\beta_{i}$ be the vertical circle containing the point $c$.

\begin{dfn}
For $\mathbf{x,y\in N}(g')$, a positive domain $p\in\pi(\mathbf{x,y})$ is said to be of type $\bm{K1}$ if it  satisfies the following conditions$\colon$
\begin{itemize}
\item $\#\{\mathbf{(x\cup y)\setminus (x\cap y)}\}\in4\mathbb{N}$.
\item Only one point of $\mathbf{(x\cup y)\setminus (x\cap y)}$ on $\beta_j$ is contained in the interior of $p$, and at the other points, three of the four adjoining squares have local multiplicity zero.
\item $p$ satisfies that $O_1(p)=X_2(p)=X_1(p)+1$, where $X_i(p)$ is the coefficient of $p$ at the square containing $X_i$ $(i=1,2)$.
\end{itemize}

A positive domain $p\in\pi(\mathbf{x,y})$ is said to be of type $\bm{K2}$ if it  satisfies the following conditions$\colon$
\begin{itemize}
\item $\#\{\mathbf{(x\cup y)\setminus (x\cap y)}\}\in4\mathbb{N}$.
\item Let $m=\#\{\mathbf{x\cup y\setminus (x\cap y)}\}$. Then, $(m-4)/4$ points of $\mathbf{(x\cup y)\setminus (x\cap y)}$ are contained in the interior of $p$, and at the other points three of the four adjoining squares have local multiplicity zero.
\item $p$ satisfies that $O_1(p)=X_2(p)=X_1(p)+1$
\end{itemize}

A positive domain $p\in\pi(\mathbf{x,y})$ is said to be of type $\bm{K3}$ if it  satisfies the following conditions$\colon$
\begin{itemize}
\item $\#\{\mathbf{(x\cup y)\setminus (x\cap y)}\}\in4\mathbb{N}$.
\item Let $m=\#\{\mathbf{x\cup y\setminus (x\cap y)}\}$. Then, $(m-4)/4$ points of $\mathbf{(x\cup y)\setminus (x\cap y)}$ are contained in the interior of $p$, and at the other points three of the four adjoining squares have local multiplicity zero.
\item One of the vertical segments of $\partial(p)$ is on the vertical circle $\beta_{j+1}$ and it goes up.
\item $p$ satisfies that $O_1(p)=X_2(p)=X_1(p)$
\end{itemize}

The set of domains of type $K1$, $K2$, and $K3$ from $\mathbf{x}$ to $\mathbf{y}$ is denoted $\pi^{K}(\mathbf{x,y})$ .
\end{dfn}
Note that however there are some domains that are both type $K1$ and type $K2$, the proof holds.

\begin{dfn}
\label{dfn:d}
The $\mathcal{R}$-module homomorphisms $\mathcal{D}^{iL}\colon CF^{-H}(g')\to CF^{-H}(g)[U_1]\llbracket1,1\rrbracket$ and $\mathcal{D}^{iR}\colon CF^{-H}(g')\to CF^{-H}(g)[U_1]$ are given by
\begin{align}
\mathcal{D}^{iL}(\mathbf{x})=\sum_{\mathbf{y\in I}(g')}\left(\sum_{p\in\pi^{iL}(\mathbf{x,y})}U_2^{O_2(p)}U_3^{O_3(p)}\dots U_n^{O_n(p)} \right)\cdot e(\mathbf{y}),\\
\mathcal{D}^{iR}(\mathbf{x})=\sum_{\mathbf{y\in I}(g')}\left(\sum_{p\in\pi^{iR}(\mathbf{x,y})}U_2^{O_2(p)}U_3^{O_3(p)}\dots U_n^{O_n(p)} \right)\cdot e(\mathbf{y}).
\end{align}
Then, $\mathcal{D}\colon CF^{-H}(g')\to\mathrm{Cone}(U_1-U_2)$ is defined by
\[
\mathcal{D}(\mathbf{x})=(\mathcal{D}^{iL}(\mathbf{x}),\mathcal{D}^{iR}(\mathbf{x})).
\]
\end{dfn}

\begin{dfn}
\label{dfn:s}
The $\mathcal{R}$-module homomorphisms $\mathcal{S}^{oL}\colon CF^{-H}(g)[U_1]\llbracket1,1\rrbracket \to CF^{-H}(g')$ and $\mathcal{S}^{oR}\colon CF^{-H}(g)[U_1]\to CF^{-H}(g')$ are given by
\begin{align}
\mathcal{S}^{oL}(\mathbf{x})=\sum_{\mathbf{y\in S}(g')}\left(\sum_{p\in\pi^{oL}(\mathbf{e^{-1}(x),y})}U_2^{O_2(p)}U_3^{O_3(p)}\dots U_n^{O_n(p)} \right)\cdot \mathbf{y},\\
\mathcal{S}^{oR}(\mathbf{x})=\sum_{\mathbf{y\in S}(g')}\left(\sum_{p\in\pi^{oR}(\mathbf{e^{-1}(x),y})}U_2^{O_2(p)}U_3^{O_3(p)}\dots U_n^{O_n(p)} \right)\cdot \mathbf{y}.
\end{align}
Then, $\mathcal{S}\colon \mathrm{Cone}(U_1-U_2)\to CF^{-H}(g')$ is defined by
\[
\mathcal{S}(\mathbf{x})=(\mathcal{S}^{oL}(\mathbf{x}),\mathcal{S}^{oR}(\mathbf{x})).
\]
\end{dfn}

Let $\partial'$ denote the differential of $CF^{-H}(g')$.
\begin{prop}
The maps $\mathcal{D},\mathcal{S}$ are chain maps.
\end{prop}
\begin{proof}
From \cite[Lemma 13.3.13]{grid-book}, $\mathcal{D}$ is a chain map by pairing domains of $\mathcal{D}\circ\partial'+\partial_{Cone}\circ\mathcal{D}$.

Consider the bijection $\theta\colon\mathbf{S}(g')\to\mathbf{S}(g')$ determined by the $180^\circ$ rotation of the horizontal circle containing $c$.
Obviously $\theta$ induces a natural bijection $\theta(\mathbf{x,y})\colon\pi\mathbf({x,y})\to\pi(\mathbf{\theta(y),\theta(x)})$ for each $\mathbf{x,y\in S}(g')$.
Then a composite domain $d*s\ (d\in\pi(\mathbf{x,y}),\ s\in\pi(\mathbf{y,z}))$ appears in $\mathcal{D}\circ\partial'+\partial_{Cone}\circ\mathcal{D}$ if and only if $s'*d'\ (s'\in\pi(\mathbf{\theta(z),\theta(y)}),\ d'\in\pi(\mathbf{\theta(y),\theta(x)}))$ appears in $\mathcal{S}\circ\partial_{Cone}+\partial'\circ\mathcal{S}$.
Therefore we can also pair all domains of $\mathcal{S}\circ\partial_{Cone}+\partial'\circ\mathcal{S}$.
\end{proof}

\begin{dfn}
\label{dfn:k}
The $\mathcal{R}$-module homomorphisms $\mathcal{K}\colon CF^{-H}(g')\to CF^{-H}(g')$ is given by
\begin{align}
\mathcal{K}(\mathbf{x})=\sum_{\mathbf{y\in S}(g')}\left(\sum_{p\in\pi^{K}(\mathbf{x,y})}U_2^{O_2(p)}U_3^{O_3(p)}\dots U_n^{O_n(p)} \right)\cdot \mathbf{y}.
\end{align}
\end{dfn}

\begin{dfn}
The \textbf{complexity} of a domain is one if it is the trivial domain and, otherwise, is the number of horizontal segments in its boundary.
Let us denote by $k(p)$ the complexity of $p$.
\end{dfn}

\begin{prop}
The above homomorphisms satisfy that
\begin{equation}
\mathcal{D}\circ\mathcal{S}=\mathrm{id},\label{ds1} \\
\end{equation}
\begin{equation}
\mathcal{S}\circ\mathcal{D}+\partial'\circ\mathcal{K}+\mathcal{K}\circ\partial'=\mathrm{id}.\label{ds2}
\end{equation}
\end{prop}

\begin{proof}
The map $D^{iL}$ (respectively $S^{oL}$) can be decomposed as $D^{iL}=D^{iL}_1+D^{iL}_{>1}$ (respectively $S^{oL}=S^{oL}_1+S^{oL}_{>1}$), where the subscript represents the restriction on the complexity of the domains.
Then we can draw the following diagram:
\[
\xymatrix@R=28pt{
CF^{-H}(g)[U_1][1,1] \ar[r]^-{U_1-U_2} \ar[d]_-{S^{oL}_1} \ar[dr]_-{S^{oL}_{>1}} & CF^{-H}(g)[U_1] \ar[d]^-{S^{oR}} \\
\mathrm{I} \ar[d]_-{D^{\mathrm{iL}}_1} \ar@{-}[r] &\mathrm{N}\ar[ld]_-{D^{\mathrm{iL}}_{>1}} \ar[d]^-{D^{\mathrm{iR}}} \\
CF^{-H}(g)[U_1][1,1] \ar[r]^-{U_1-U_2}  & CF^{-H}(g)[U_1]
}
\]
Note that the complexities of domains of $\pi^{iL}$ and $\pi^{oL}$ are odd, and ones of $\pi^{iL}$ and $\pi^{oL}$ are even.

It is convenient to write the horizontal circle crossing the point $c$ as $\alpha_i$ and the vertical circle crossing the point $c$ as $\beta_j$.

Let us first examine the equation (\ref{ds1}).
To see the equation (\ref{ds1}), we will count domains of the left side of (\ref{ds1}) and check that domains except for the trivial ones are canceled in modulo 2.

Since we have obviously $D^{iL}_1\circ S^{oL}_1=\mathrm{id}$, all we need to check is $D^{iR}\circ S^{oR}=\mathrm{id}$, $\mathcal{D}\circ S^{oL}_{>1}=0$, and $\mathcal{D}\circ S^{oR}=0$.
Let $s*d$ be a composite domain with $s\in\pi^{oR}(\mathbf{x,y})\cup\pi^{oL}(\mathbf{x,y}),\ d\in\pi^{iR}(\mathbf{y,z})\cup\pi^{iL}(\mathbf{y,z})$ appearing in $D^{iR}\circ S^{oR}$ 
, $\mathcal{D}\circ S^{oL}_{>1}$, or $\mathcal{D}\circ S^{oR}$.
\begin{enumerate}[(1)]
\item If $k(s)=k(d)=2$, then $s*d$ appears in $D^{iR}\circ S^{oR}$ and we get $\mathbf{x=z}$.
\item Otherwise, consider two points $y_1=\mathbf{y}\cap\alpha_i$ and $y_2=\mathbf{y}\cap\beta_j$.
We see that $y_1,y_2\notin\mathbf{z}$ since $c\in\mathbf{z}$.
Then $d$ must overlay $s$.
If $k(s)=3$ and $k(d)=3$, at some of the corner points of $d,s$, all adjoining squares have a local multiplicity greater than $0$.
Therefore such a domain cannot be taken.
\item If $k(s)>3$ or $k(d)>3$, then there are some points of the initial state as a point that is not a corner in the interior of $d$ or $s$.
Again, such a domain cannot be taken.
\end{enumerate}
These arguments prove that $D^{iR}\circ S^{oR}=\mathrm{id}$, $\mathcal{D}\circ S^{oL}_{>1}=0$, and $\mathcal{D}\circ S^{oR}=0$.

To see the equation (\ref{ds2}), consider the following diagram$\colon$
\[
\xymatrix@R=28pt{
\mathrm{I} \ar[d]_-{D^{\mathrm{iL}}_1} \ar@{-}[r] &\mathrm{N}\ar[ld]_-{D^{\mathrm{iL}}_{>1}} \ar[d]^-{D^{\mathrm{iR}}} \\
CF^{-H}(g)[U_1][1,1] \ar[r]^-{U_1-U_2} \ar[d]_-{S^{oL}_1} \ar[dr]_-{S^{oL}_{>1}} & CF^{-H}(g)[U_1] \ar[d]^-{S^{oR}} \\
\mathrm{I}  \ar@{-}[r] &\mathrm{N}
}
\]
For $\mathbf{x\in S}(g'))$, let $p=d*s\ (d\in\pi(\mathbf{x,y}),\ s\in\pi(\mathbf{y,z}))$ be a composite domain appearing in $\mathcal{S}\circ\mathcal{D}$.
Now we will observe that each $p$ appears also in $\partial'\circ\mathcal{K}+\mathcal{K}\circ\partial'$ or $\mathrm{id}$.
We have five cases (figure \ref{fig:sd2}-\ref{fig:sd5})$\colon$

\begin{figure}
\centering
\includegraphics[scale=0.6]{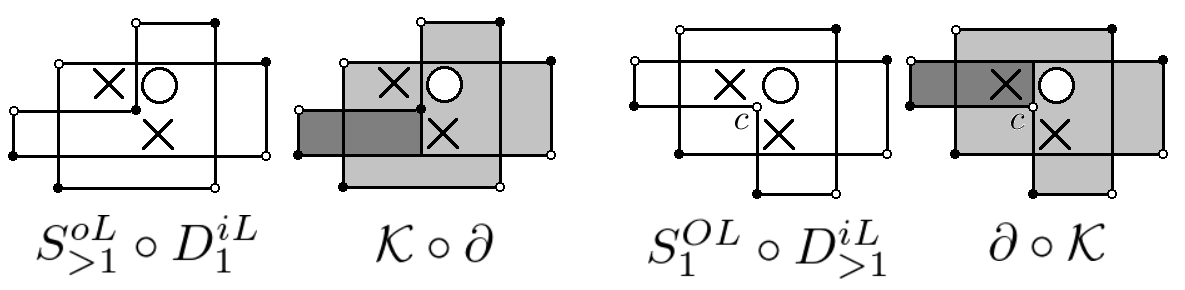}
\caption{Examples of the case (ii)}
\label{fig:sd2}
\includegraphics[scale=0.6]{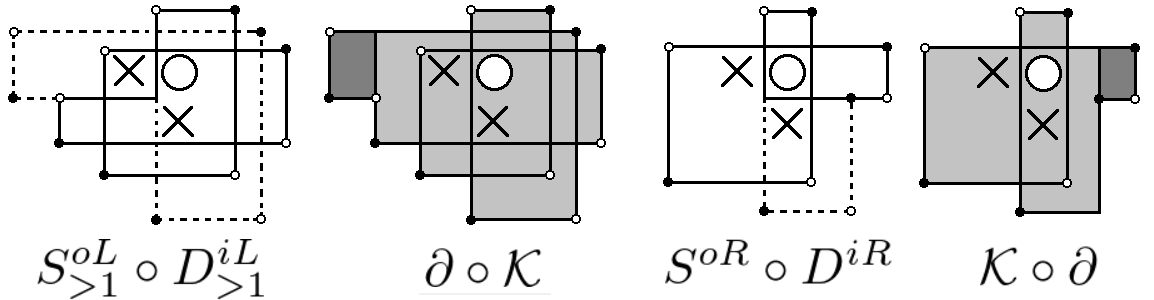}
\caption{Examples of the case (iii)}
\label{fig:sd3}
\includegraphics[scale=0.6]{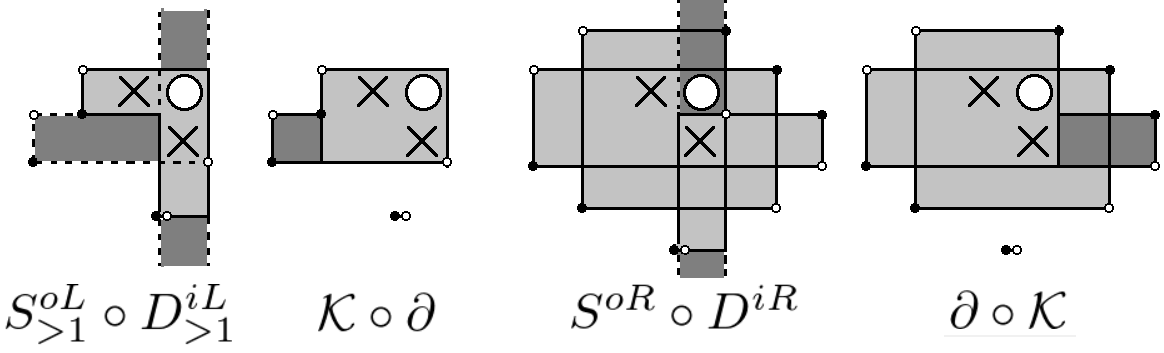}
\caption{Examples of the case (iv)}
\label{fig:sd4}
\includegraphics[scale=0.6]{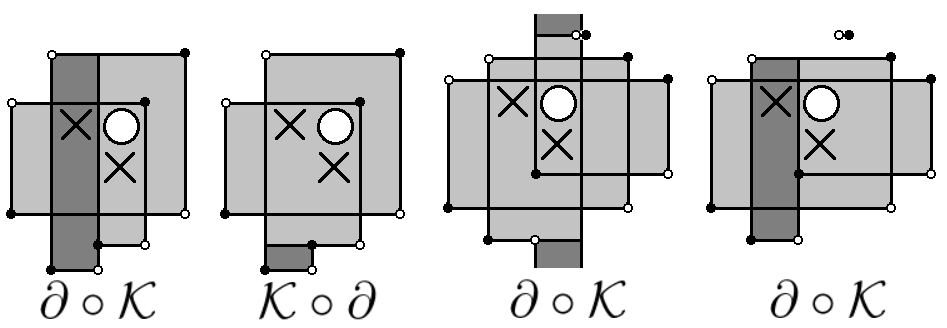}
\caption{Examples of the case (v)}
\label{fig:sd5}
\end{figure}
\end{proof}

\begin{enumerate}[(i)]
\item If $k(d)=k(s)=1$, then the composite domain $d*s$ must be trivial. This domain also appears in $\mathrm{id}_{CF^{-H}(g')}$.
\item If $k(d)=1$ or $k(s)=1$, then we can take a small square $r$ from $p$ by cutting $p$ along $\beta_j$ and another composition $p=(p-r)*r$ or $p=r*(p-r)$ appearing in $\partial'\circ\mathcal{K}$ or $\mathcal{K}\circ\partial'$.
In this case, $p-r$ is of type $K1$.
\item If $k(d)>1$ and $k(s)>1$ and $p$ does not contain any vertical thin annulus, then there is a $270^\circ$ corner at one point $x_l$ of $\mathbf{(x\cup y)\setminus(x\cap y)}$.
Cutting $p$ along the horizontal or vertical circle containing $x_l$ gives another decomposition using the domain of type $K2$.
\item If $k(d)>1$ and $k(s)>1$ and $p$ has the vertical thin annulus $p'$ containing $O_1$, then consider the domain $p-p'$.
This domain connects the same states as $p$ and can be decomposed using domains of type $K3$.
If it is decomposed as $p-p=r*k$ or $p-p=k*r$, the rectangle $r$ does not contain $O_1$, so they are canceled in modulo 2.
\item Finally, we need to check that the remaining domains in $\partial'\circ\mathcal{K}+\mathcal{K}\circ\partial'$ are canceled each other.
Let $p=r*k$ or $p=k*r$ be the remaining composite domain.
Let $m$ be the number of corners that $d$ and $s$ share, in other words, let $m=\#\{\mathbf{((x\cup y)\setminus(x\cap y))}\cap\mathbf{((y\cup z)\setminus(y\cap z))}\}$.
Then we have three cases$\colon$ $m=0,1,3$.
\begin{enumerate}[(a)]
\item If $m=0$, then $p$ can be decomposed as $p=r*k$ and $p=k*r$.
\item If $m=1$, then $p$ has a $270^\circ$ corner at one point of $\mathbf{(x\cup y)\setminus(x\cap y)}$.
Likely the case (iii), cutting at the $270^\circ$ corner in two different directions gives the two decompositions of $p$.
\item If $m=3$, then $p$ contains a vertical thin annulus $p'$ containing $O_1$, then two domains $p$ and $p-p'$ are canceled.
The domain of type $\bm{K1}$ or $\bm{K2}$ is used for the decomposition of $p$, and the domain of type $\bm{K3}$ of $p'$.
\end{enumerate}
\end{enumerate}

\section{Proof of the main theorem \ref{main} and theorem \ref{main2}}
\begin{proof}[Proof of Theorem \ref{main}]
If $g'$ is obtained from $g$ by a single cyclic permutation or commutation$'$, Proposition \ref{chainho1} and \ref{chainho2} provide $CF^{-H}(g)\simeq CF^{-H}(g')$. 
This chain homotopy equivalent gives induced chain homotopy equivalent $CF^{-H}_U(g)\simeq CF^{-H}_U(g')$ as Maslov graded chain complexes over $\mathbb{F}[U]$-module.
Proposition \ref{ct=ct} and \ref{U=t} provides $tCF^{-H}(g)\simeq tCF^{-H}(g')$.

If $g'$ is obtained from $g$ by a single stabilization$'$, then $CF^{-H}(g')\simeq \mathrm{Cone}(U_1-U_2)$ by Proposition \ref{chainho3}.
By identifying all $U_i$'s, $\mathrm{Cone}(U_1-U_2)$ turns into $CF^{-H}(g)\oplus CF^{-H}(g)\llbracket1,1\rrbracket$ because $U_1$ is homogeneous of degree $(-2,-1)$.
Then we obtain induced chain homotopy $CF^{-H}_U(g')\simeq CF^{-H}(g)\oplus CF^{-H}(g)\llbracket1,1\rrbracket$.
Proposition \ref{ct=ct} and \ref{U=t} provides $tCF^{-H}(g')\simeq tCF^{-H}(g)\oplus W_t$.
\end{proof}
\begin{proof}[Proof of Theorem \ref{main2}]
\label{propmain2}
For $t\in[0,1]$, the grading shift $\llbracket1-t\rrbracket$ from $W_t$ does not affect the value of the $\Upsilon$ because $1-t\geq0$.
\end{proof}

\section{Link cobordisms with grid homology}
In this section, we observe graph grid diagrams for two links connected by a link cobordism.
The basic idea is developed by Sarkar \cite{grid-tau-sarkar}.

First, we will use the extended grid diagram which has an $X$-marking and an $O^*$-marking in the same square.
The square containing both an $X$-marking and an $O^*$-marking represents an unknotted, unlinked component.
It is known that the homology of these grid diagrams is also invariant.
See \cite[Section 8.4]{grid-book} for details about extended grid diagrams.
A grid diagram representing a link is called \textbf{tight} if there is exactly one $O^*$-marking in each link component.

In general, the symmetrized Alexander grading is not canonical because it needs to calculate the homology of $CF^-(g)$.
However, if the balanced spatial graph is a link, we can take a tight grid diagram representing it, and the symmetrized Alexander grading is written explicitly:
from \cite[Section 8.2]{grid-book}, 
\begin{equation}
\label{def-Alexander-link}
A^H(\mathbf{x})=\mathcal{J}(\mathbf{x},\mathbb{X}-\mathbb{O})-\frac{1}{2}\mathcal{J}(\mathbb{X},\mathbb{X})+\frac{1}{2}\mathcal{J}(\mathbb{O},\mathbb{O})-\frac{n-l}{2},
\end{equation}
where $l$ is the number of link components.

According to Sarkar \cite{grid-tau-sarkar} and Vance \cite{grid-tau--Vance-spatial}, two tight grid diagrams representing two links connected by a link cobordism are connected by a finite sequence of link-grid moves.
These moves are commutation$'$s, (de-)stabilization$'$s, births, $X$-saddles, $O$-saddles, and deaths.

A grid diagram $g'$ is obtained from $g$ by a \textbf{birth} (figure \ref{fig:birth}) if adding one row and column to $g$ and putting an $X$-marking and $X$-marking in the square which is the intersection of the new row and column. A grid move \textbf{death} is the inverse move of a birth.
The move birth (respectively death) represents a birth (respectively a death) on link cobordism.
These moves are link-grid move $(4)$ (respectively $(7)$) in \cite{grid-tau-sarkar}.

A grid diagram $g'$ is obtained from $g$ by an \textbf{$X$-saddle} (figure \ref{fig:xsaddle}) if $g$ has a $2\times2$ small squares with two $X$-markings at tow-left and bottom-right, and $g'$ is obtained from $g$ by deleting these two markings and putting new ones at top-right and bottom-left.
The move $X$-saddle represents a saddle move on link cobordism.
These moves are link-grid move $(5)$ in \cite{grid-tau-sarkar}.

A grid diagram $g'$ is obtained from $g$ by \textbf{$O$-saddle} (figure \ref{fig:osaddle}) if $g$ has a $2\times2$ small squares with one $O^*$-marking at tow-left and one $O$-marking at bottom-right, and $g'$ is obtained from $g$ by deleting these two markings and putting new two $O^*$-markings at top-right and bottom-left.
The move $O$-saddle represents a split move on link cobordism.
These moves are link-grid move $(6)$ in \cite{grid-tau-sarkar}.

\begin{figure}
\centering
\includegraphics[scale=0.7]{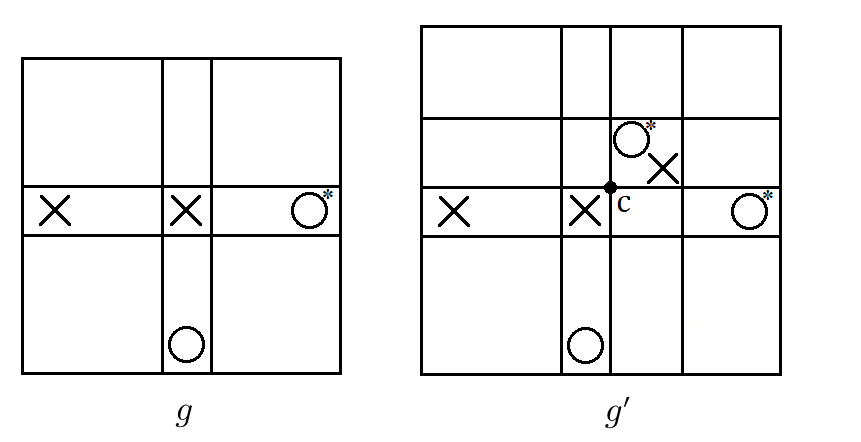}
\caption{A birth in $g$ produces $g'$, a death in $g'$ produces $g$.}
\label{fig:birth}
\includegraphics[scale=0.7]{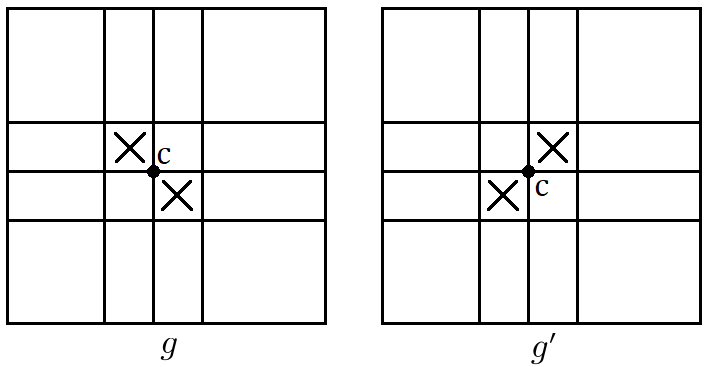}
\caption{An $X$-saddle in $g$ produces $g'$}
\label{fig:xsaddle}
\includegraphics[scale=0.7]{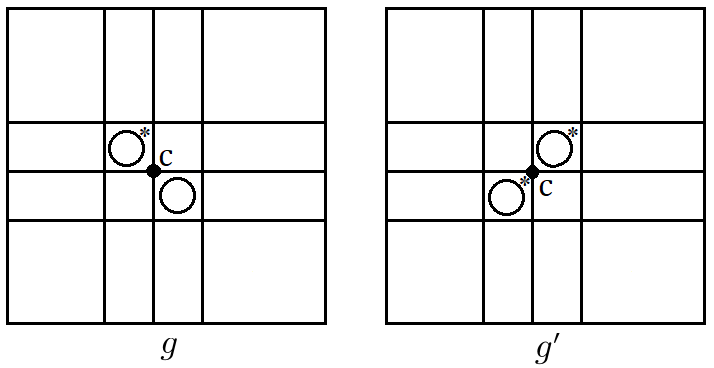}
\caption{An $O$-saddle in $g$ produces $g'$}
\label{fig:osaddle}
\end{figure}

Using grid homology for links, Sarkar \cite{grid-tau-sarkar} evaluated the maximum changes of Maslov and Alexander grading on each link-grid move.
We will check that we can define the appropriate maps also on the t-modified chain complexes and that the changes of t-grading are the same as Sarkar's evaluation with $M-t\cdot A$.

In order to prove theorem \ref{main3}, we evaluate the change of grading on each link-grid move.
It is convenient to introduce an alternative Alexander grading $A'\colon$
\begin{dfn}[{\cite[Definition 4.2]{grid-tau--Vance-spatial}}]
For $\mathbf{x\in S}(g)$,
\[
A'(\mathbf{x})=\mathcal{J}(\mathbf{x},\mathbb{X}-\mathbb{O})-\frac{1}{2}\mathcal{J}(\mathbb{X},\mathbb{X})+\frac{1}{2}\mathcal{J}(\mathbb{O},\mathbb{O})-\frac{n-1}{2}.
\]
The symmetrized Alexander grading $A^H$ can be obtained from $A'$ by adding $\frac{l-1}{2}$.
\end{dfn}
In this section, we are thinking about chain complexes $tCF^-(g)$ with t-grading using $A'$ rather than $A^H$, we set $\mathrm{gr}_t(\mathbf{x})=M(\mathbf{x})-tA'(\mathbf{x})$.

For $H(tCF^-(g))=tHF^-(g)$, we can define $\Upsilon'_g(t)$ in the same way as $\Upsilon_g(t)$.
It is clear that 
\begin{equation}
\label{up'up}
\Upsilon_g(t)=\Upsilon'_g(t)-\frac{l-1}{2}t.
\end{equation}

\subsection{births and deaths}
Proposition \ref{prop:birth} will imply that there are isomorphisms
\begin{align*}
D'\colon tHF^{-}(g')\to tHF^{-}(g)\oplus tHF^-(g)\llbracket1\rrbracket,\\
S'\colon tHF^{-}(g)\oplus tHF^-(g)\llbracket1\rrbracket \to tHF^{-}(g').
\end{align*}
These maps preserve t-grading if we use symmetrized Alexander grading $A^H$, so $D'$ shifts t-grading by $-\frac{1}{2}t$ and $S'$ by $\frac{1}{2}t$.
We can write these maps as $D'=(H((\mathcal{D}^{iR})^t),H((\mathcal{D}^{iL})^t))$ and $S'=(H((\mathcal{S}^{oR})^t),H((\mathcal{S}^{oL})^t))$ using the notations in subsection \ref{sub:sta} and Proposition \ref{ct=ct}.

As we regard $H(\mathcal{D}^{iL})$ as the map into $tHF^-(g)$ rather than $tHF^-(g)\llbracket1\rrbracket$, $H(\mathcal{D}^{iL})$ is surjective map and shifts t-grading by $1-\frac{1}{2}t$.
Then the maximum shift of t-grading of the homogeneous, non-torsion element in homology by a death is $1-\frac{1}{2}t$, which implies $\Upsilon'_g(t)\geq \Upsilon'_{g'}(t)+1-\frac{1}{2}t$.

On the other hand, the maximum change by a birth is $\frac{1}{2}t$.
Then we get $\Upsilon'_{g'}(t)\geq \Upsilon'_{g}(t)+\frac{1}{2}t$.

\subsection{$X$-saddles and $O$-saddles}
Before observing the saddles, we prepare an algebraic lemma.
Let $A$ be a $\mathcal{R}$-module.
Then the torsion submodule of $A$ is
\[
\mathrm{Tors}(A)=\{a\in A|\mathrm{there\ is\ a\ non-zero}\ p\in\mathcal{R}\ \mathrm{with}\ p\cdot a=0\}.
\]
\begin{lem}
\label{torsion}
Let $A,B$ be two $\mathcal{R}$-modules.
If $\alpha\colon A\to B$ and $\beta\colon B\to A$ are two module maps with the property that $\beta\circ\alpha=v^s$ for some $s\geq0$, then $\alpha$ induces an injective map from $A/\mathrm{Tors}(A)$ into $B/\mathrm{Tors}(B)$.
\end{lem}
\begin{proof}
If $\alpha(a)\in\mathrm{Tors}(B)$, then there is a non-zero element $p\in \mathcal{R}$ with $p\cdot\alpha(a)=0$.
Then $\beta(p\cdot\alpha(a))=v^{s}p\cdot a=0$, so $a\in\mathrm{Tors}(A)$.
\end{proof}
Then we observe $X$-saddles and $O$-saddles.

\begin{prop}
\label{x-saddle}
If $g'$ is obtained by an $X$-saddle, there are $\mathcal{R}$-module maps
\begin{align*}
\sigma\colon tHF^{-}(g)\to tHF^{-}(g'),\\
\mu\colon tHF^{-}(g') \to tHF^{-}(g).
\end{align*}
with the following properties$\colon$
\begin{itemize}
\item $\sigma$ shifts t-grading by $-\frac{1}{2}t$,
\item $\mu$ shifts t-grading by $-\frac{1}{2}t$,
\item $\mu\circ\sigma=v^t$,
\item $\sigma\circ\mu=v^t$.
\end{itemize}
\end{prop}
\begin{proof}
Let $c$ be the point in the center of the $2\times2$ squares as in Figure \ref{fig:xsaddle}.
Define $\sigma\colon tCF^{-}(g)\to tCF^{-}(g')$ and $\mu\colon tCF^{-}(g) \to tCF^{-}(g')$ by
\begin{center}
$\sigma(\mathbf{x})=
\begin{cases}
\mathbf{x} & (c\in\mathbf{x})\\
v^t\cdot\mathbf{x} & (c\notin\mathbf{x}).
\end{cases}
$
and
$\mu(\mathbf{x})=
\begin{cases}
v^t\cdot\mathbf{x} & (c\in\mathbf{x})\\
\mathbf{x} & (c\notin\mathbf{x}).
\end{cases}
$
\end{center}
It is straightforward to see that these two maps preserve Maslov grading and increase Alexander grading by $\frac{1}{2}$, so they shift t-grading by $-\frac{1}{2}t$.
Obviously, both $\mu\circ\sigma$ and $\sigma\circ\mu$ are multiplication by $v^t$.
It is also clear that both $\sigma$ and $\mu$ are chain maps.
Think of the maps on homology induced by them.
\end{proof}

By Lemma \ref{torsion}, Proposition \ref{x-saddle} means that the shift of t-grading of the homogeneous, non-torsion element in homology by a $X$-saddle is $-\frac{1}{2}t$

\begin{prop}
\label{o-saddle}
If $g'$ is obtained by an $O$-saddle, there are $\mathcal{R}$-module maps
\begin{align*}
\sigma\colon tHF^{-}(g)\to tHF^{-}(g'),\\
\mu\colon tHF^{-}(g') \to tHF^{-}(g).
\end{align*}
with the following properties$\colon$
\begin{itemize}
\item $\sigma$ shifts t-grading by $-1+\frac{1}{2}t$,
\item $\mu$ shifts t-grading by $-1+\frac{1}{2}t$,
\item $\mu\circ\sigma=v^{2-t}$,
\item $\sigma\circ\mu=v^{2-t}$.
\end{itemize}
\end{prop}
\begin{proof}
We can prove this in the same way as Proposition \ref{x-saddle}.;
All we need is to consider $\sigma\colon tCF^{-}(g)\to tCF^{-}(g')$ and $\mu\colon tCF^{-}(g) \to tCF^{-}(g')$ as
\begin{center}
$\sigma(\mathbf{x})=
\begin{cases}
\mathbf{x} & (c\in\mathbf{x})\\
v^{2-t}\cdot\mathbf{x} & (c\notin\mathbf{x}).
\end{cases}
$
and
$\mu(\mathbf{x})=
\begin{cases}
v^{2-t}\cdot\mathbf{x} & (c\in\mathbf{x})\\
\mathbf{x} & (c\notin\mathbf{x}).
\end{cases}
$
\end{center}
By definitions, these two maps drop Maslov grading by $1$ and Alexander grading by $\frac{1}{2}$, so they shift t-grading by $-1+\frac{1}{2}t$.
It is clear that $\mu\circ\sigma$ and $\sigma\circ\mu$ are multiplication by $v^{2-t}$.

\end{proof}
Again using Lemma \ref{torsion}, Proposition \ref{o-saddle} means that the shift of t-grading of the homogeneous, non-torsion element in homology by an $O$-saddle is $-1+\frac{1}{2}t$.

\section{Proof of the main Theorem \ref{main3}}
We verify Theorem \ref{main3} using the same method as Sarkar \cite{grid-tau-sarkar} and Vance \cite{grid-tau--Vance-spatial}.
\begin{prop}[{\cite[Theorem 4.1]{grid-tau-sarkar}}]
\label{link-grid move}
If $g_1,g_2$ are two tight grid diagrams representing $l_1,l_2$-component links $L_1, L_2$, respectively, and if there is a link cobordism with $b$ births, $s$ saddles, and $d$ deaths, then there is a sequence of link-grid moves connecting from $g_1$ to $g_2$, such that there are exactly $b$ births, $s-d+l_1-l_2$ $X$-saddles, $d-l_1+l_2$ $O$-saddles, and $d$ deaths, and these happen in this order.
\end{prop}

\begin{proof}[Proof of Theorem \ref{main3}]
Let $g_1,g_2$ be two grid diagrams represents $L_1,L_2$, respectively.
By Proposition \ref{link-grid move}, there is a sequence of link-grid moves taking $g_1$ to $g_2$ such that $b$ births, $s-d+l_1-l_2$ $X$-saddles, $d-l_1+l_2$ $O$-saddles, and $d$ deaths happen in this order.
Take a homogeneous, non-torsion element $\alpha\in tHF^-(g_1)$ whose t-grading is $\Upsilon'_{g_1}(t)$.
Composing maps in the previous section, we get a map that sends $\alpha$ onto $tHF^-(g_2)$.
By adding up the t-grading shifts of each link-grid move, then we get
\begin{align*}
\Upsilon'_{g_1}(t)+b\cdot\frac{1}{2}t+(s-d+l_1-l_2)\cdot\left(-\frac{1}{2}t\right)+(d-l_1+l_2)\cdot\left(-1+\frac{1}{2}t\right)+d\cdot\left(1-\frac{1}{2}t\right)\\
\leq\Upsilon'_{g_2}(t),
\end{align*}
so
\begin{align*}
\Upsilon'_{g_1}(t)-\frac{l_1-1}{2}t-\left(\frac{1}{2}(s-b-d)+1-\frac{l_1+l_2}{2}\right)t-(l_1-1)t+l_1-l_2\\
\leq\Upsilon'_{g_2}(t)-\frac{l_2-1}{2}t.
\end{align*}
Use (\ref{up'up}) and $g=\frac{1}{2}(s-b-d)+1-\frac{l_1-l_2}{2}$, then
\[
\Upsilon_{L_1}(t)-tg-t(l_1-1)-(l_1-l_2)\leq\Upsilon_{L_2}(t)
\]
We can once get the other inequality if we reverse the direction of link cobordism, so we see that
\[
\Upsilon_{L_1}(t)-tg-t(l_1-1)-(l_1-l_2)\leq\Upsilon_{L_2}(t)\leq\Upsilon_{L_1}(t)+tg+t(l_2-1)+(l_2-l_1).
\]
\end{proof}

\section{Proof of some properties of the $\Upsilon$ invariant}

\subsection{The value of $\Upsilon$ at $t=0$}
\begin{proof}[Proof of Proposition \ref{prop:0}]
When $t=0$, the t-modified chain complex $tCF^{-H}(g)$ is independent of the $X$-markings because the differential is
\[
\partial_t^-(\mathbf{x})=\sum_{\mathbf{y}\in\mathbf{S}(g)}\left(
\sum_{r\in \mathrm{Rect}^\circ(\mathbf{x,y})}v^{2|\mathbb{O}\cap r|}
\right)\mathbf{y},
\]
and t-grading is
\[
\mathrm{gr}_t(v^\alpha\mathbf{x})=M(\mathbf{x})-\alpha.
\]
Also, we do not need to distinguish between $O$-markings and $O^*$-markings.
Let $n$ be a size of $g$.
Remove all $X$-markings of $g$ and put $n$ $X$-markings so that the new tight grid diagram $g'$ represents an unknot.
Then $HF^{-H}(g')\cong\mathbb{F}[U]$ in grading zero and $tCF^{-H}(g)\cong tCF^{-H}(g')$ as chain complexes.
\begin{figure}
\centering
\includegraphics[scale=0.7]{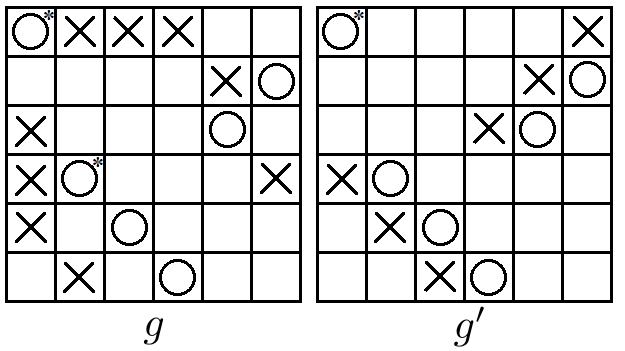}
\caption{Getting grid diagram representing an unknot}
\end{figure}
By Proposition \ref{ct=ct}, the universal coefficient theorem, and \cite[Lemma 14.1.11]{grid-book}, we have $tHF^{-H}(g')\cong H\left((\frac{CF^{-H}(g')}{U_1=\dots=U_n})\otimes_{\mathbb{F}[U]}\mathcal{R}_0\right)\cong\mathcal{R}_0\otimes W_0^{n-1}$, where $\mathcal{R}_0$ is the t-graded module isomorphic to $\mathcal{R}$ in grading zero and $W_0\cong\mathbb{F}_{0}\oplus\mathbb{F}_{-1}$ (in other words, $t=0$ for $W_t$).
Similar to the proof of Theorem \ref{main2}, the grading shift $\llbracket1\rrbracket$ does not affect the value of the $\Upsilon$.
Therefore $\Upsilon_g(0)=\Upsilon_{g'}(0)=0$.
\end{proof}

\subsection{Crossing change}
\begin{prop}
\label{cc}
If two graph grid diagrams $g_+$ and $g_-$ represent two links $L_+$ and $L_-$ that differ in a crossing change, then there are $\mathcal{R}$ maps
$C^t_-\colon tHF^{-H}(g_+)\to tHF^{-H}(g_-)$ and $C^t_+\colon tHF^{-H}(g_-)\to tHF^{-H}(g_+)$,
with the following properties$\colon$
\begin{itemize}
\item $C_-$ is graded,
\item $C_+$ shifts t-grading by $-2+t$,
\item $C_-\circ C_+=v^{2-t}$,
\item $C_+\circ C_-=v^{2-t}$,
\end{itemize}
\end{prop}
\begin{proof}
Combine \cite[Proposition 6.1.1]{grid-book} and Proposition \ref{ct=ct}.
\end{proof}
\begin{proof}[Proof of Proposition \ref{prop:1}]
We can show this in the same argument as \cite[Theorem 5.10]{Foldvari-grid-upsilon} by using maps of Proposition \ref{cc}.
\end{proof}

\subsection{Adding an unknot}
\label{adding-unknot}
\begin{prop}
\label{prop:birth}
If $g,g'$ be two graph grid diagrams as in Figure \ref{fig:birth}, then as graded $\mathcal{R}$-modules,
\begin{equation}
\label{t-birth}
tHF^{-H}(g')\cong tHF^{-H}(g)\oplus tHF^{-H}(g)\llbracket1\rrbracket,
\end{equation}
where $\llbracket1\rrbracket$ is t-grading shift by $1$.
\end{prop}
\begin{proof}
The basic idea is the same as in \cite[Section 8.4]{grid-book}, in other words, we can use the same maps as the maps for stabilization$'$ invariance.

We assume that $CF^{-H}(g)$ is a chain complex over $\mathbb{F}[U_2,\dots,U_n]$ and $CF^{-H}(g')$ is one over $\mathbb{F}[U_1,\dots,U_n]$.

We will construct a chain homotopy equivalence
\begin{equation}
\label{eq:addingunknot}
CF^{-H}(g')\simeq CF^{-H}(g)[U_1]\oplus CF^{-H}(g)[U_1]\llbracket1,0\rrbracket.
\end{equation}

As in Figure \ref{fig:birth}, we assume that $g'$ has $2\times2$ squares such that the top-right square contains both one $O$-marking and one $X$-marking and the bottom-left square has one $O$- or $O^*$-marking.
By definition, $g'$ is strictly not a graph grid diagram but we can consider about $CF^{-H}(g')$ in the same manner (See \cite[Lemma 8.4.2]{grid-book} for detail).
We denote by c the intersection point of the new horizontal and vertical circles in $g'$
Under these settings, we can use the same maps $\mathcal{D}$, $\mathcal{S}$, and $\mathcal{K}$ as in Definition \ref{dfn:d}, \ref{dfn:s}, and \ref{dfn:k} respectively for $CF^{-H}(g)$ and $CF^{-H}(g')$.
Direct computations show that the grading changes of these maps are different from the maps of stabilization$'$ invariance in Section \ref{sub:sta}$\colon$ $\mathcal{D}^{iR}$ is bigraded,  $\mathcal{D}^{iL}$ shifts Maslov grading by $1$, $\mathcal{D}^{oR}$ is bigraded, and $\mathcal{D}^{oL}$ shifts Maslov grading by $-1$.
Counting domains are independent of markings, so $\mathcal{D}$, $\mathcal{S}$ are chain homotopy equivalences.

Applying Proposition \ref{ct=ct} to the induced chain homotopy equivalence $CF^{-H}_U(g')\simeq CF^{-H}_U(g)\oplus CF^{-H}_U(g)\llbracket1,0\rrbracket$ from (\ref{eq:addingunknot}), we get (\ref{t-birth}).

\end{proof}

\begin{proof}[Proof of \ref{prop:2}]
It is immediately from (\ref{t-birth}) in Proposition \ref{prop:birth}.
Similar to the proof of Theorem \ref{main2}, the grading shift $\llbracket1\rrbracket$ does not affect the value of the $\Upsilon$.
\end{proof}

\subsection{Wedge sum of an unknot}
\begin{prop}
\label{prop:wedge}
If $g,g'$ be two graph grid diagrams as in Figure \ref{fig:wedge}, then as graded $\mathcal{R}$-modules,
\begin{equation}
\label{t-wedge}
tHF^{-H}(g')\cong tHF^{-H}(g)\otimes W_t.
\end{equation}
\end{prop}
\begin{proof}
The same argument as Proposition \ref{chainho3} works even if there is one extra $O^*$-marking in the $2\times 2$ block because the domains that appear in $\mathcal{D}$, $\mathcal{S}$, and $\mathcal{K}$ is independent of the extra $O^*$-marking.
We can show that there is a chain homotopy equivalence
\[
CF^{-H}(g')\simeq \mathrm{Cone}(U_1-U_2\colon CF^{-H}(g)[U_1]\to CF^{-H}(g)[U_1]),
\]
where we regard $CF^{-H}(g')$ as $\mathbb{F}[U_1,\dots,U_n]$-module and $CF^{-H}(g)$ as $\mathbb{F}[U_2,\dots,U_n]$-module.
\begin{figure}
\centering
\includegraphics[scale=0.7]{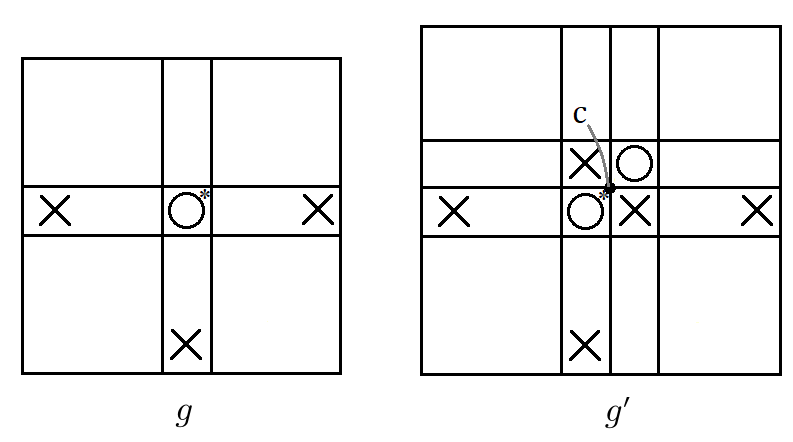}
\caption{grid diagram representing $f$, $f\#_c\mathcal{O}$}
\label{fig:wedge}
\end{figure}
Applying Proposition \ref{ct=ct} to the induced chain homotopy equivalence $CF^{-H}_U(g')\simeq CF^{-H}_U(g)\oplus CF^{-H}_U(g)\llbracket1,1\rrbracket$ by the relation $U_1=\dots=U_n$, we get (\ref{t-wedge}).
\end{proof}

\begin{proof}[Proof of \ref{prop:3}]
It is immediately from Proposition \ref{prop:wedge}.
Similar to the proof of Theorem \ref{main2} and \ref{prop:2}, the grading shift $\llbracket1-t\rrbracket$ from $W_t$ does not affect the value of the $\Upsilon$ for $t\in[0,1]$ because $1-t\geq0$.
\end{proof}

\section*{Acknowledgements}
I would like to express my sincere gratitude to my supervisor, Tetsuya Ito, for useful discussions and corrections.
I sincerely thank the anonymous reviewers for their valuable feedback.
This work was supported by JST, the establishment of university fellowships towards
the creation of science technology innovation, Grant Number JPMJFS2123.

\bibliography{grid}
\bibliographystyle{amsplain} 

\end{document}